\newif\ifnips
\nipsfalse

\ifnips
\documentclass{article}
\else
\documentclass[11pt]{article}
\fi

\usepackage{amsmath}
\usepackage{amssymb}
\usepackage{amsthm}
\usepackage{amsfonts}
\usepackage{dsfont}
\usepackage{cite}
\usepackage{enumitem}
\usepackage{bm}
\usepackage[colorinlistoftodos]{todonotes}
\usepackage{mathtools}
\usepackage{algorithm}
\usepackage[noend]{algpseudocode}
\usepackage{verbatim}
\usepackage{subfigure}
\usepackage{calc}

\usepackage{geometry}
\newcommand{\innp}[1]{\left\langle #1 \right\rangle}

\newcommand{\mA}{\mathbf{A}}

\newcommand{\vx}{\mathbf{x}}
\newcommand{\vxh}{\mathbf{\hat{x}}}

\newcommand{\vzh}{\mathbf{\hat{z}}}
\newcommand{\vy}{\mathbf{y}}
\newcommand{\vz}{\mathbf{z}}

\newcommand{\vb}{\mathbf{b}}

\newcommand{\vu}{\mathbf{u}}

\newcommand{\defeq}{\stackrel{\mathrm{\scriptscriptstyle def}}{=}}
\newcommand{\etal}{\textit{et al}.}

\newenvironment{myalign*}{\par\nobreak\ifnips\footnotesize\else\small\fi\noindent\align*}{\endalign*}

\graphicspath{{imgs/}}

\makeatletter
\def\mathcolor#1#{\@mathcolor{#1}}
\def\@mathcolor#1#2#3{%
  \protect\leavevmode
  \begingroup
    \color#1{#2}#3%
  \endgroup
}
\makeatother

\newcommand*{\vsepfbox}[1]{%
  \begingroup
    \sbox0{\fbox{#1}}%
    \setlength{\fboxrule}{0pt}%
    \mbox{\kern-\fboxsep\fbox{\unhbox0}\kern-\fboxsep}%
  \endgroup
}

\theoremstyle{plain} \numberwithin{equation}{section}
\newtheorem{theorem}{Theorem}[section]
\numberwithin{theorem}{section}

\newtheorem{lemma}[theorem]{Lemma}
\newtheorem{proposition}[theorem]{Proposition}

\newtheorem{fact}[theorem]{Fact}
\theoremstyle{definition}
\newtheorem{definition}[theorem]{Definition}

\usepackage{titling}
\makeatletter
\newtheorem*{rep@theorem}{\rep@title}
\newcommand{\newreptheorem}[2]{%
\newenvironment{rep#1}[1]{%
 \def\rep@title{#2 \ref{##1}}%
 \begin{rep@theorem}}%
 {\end{rep@theorem}}}
\makeatother
\newreptheorem{lemma}{Lemma}
\newreptheorem{proposition}{Proposition}

\newcommand{\axgd}{\textsc{axgd}\xspace}
\newcommand{\agd}{\textsc{agd}\xspace}
\newcommand{\gd}{\textsc{gd}\xspace}
\newcommand{\amd}{\textsc{amd}\xspace}

\ifnips
 \usepackage[nonatbib, final]{nips_2017}
\fi

\usepackage[utf8]{inputenc} 
\usepackage[T1]{fontenc}    
\usepackage{hyperref}       
\usepackage{url}            
\usepackage{booktabs}       
\usepackage{amsfonts}       
\usepackage{nicefrac}       
\usepackage{microtype}      
\usepackage{xspace}

\title{Accelerated Extra-Gradient Descent:\\ A Novel Accelerated First-Order Method\thanks{Part of this work was done while the authors were visiting the Simons Institute for the Theory of Computing. It was partially supported by NSF grant \#CCF-1718342 and by the DIMACS/Simons Collaboration on Bridging Continuous and Discrete Optimization through NSF grant \#CCF-1740425.}}

\author{Jelena Diakonikolas and Lorenzo Orecchia\\
Computer Science Department, 
Boston University\\
email: \{jelenad, orecchia\}@bu.edu }
\date{}

\begin{document}
 
\maketitle

\begin{abstract}
We provide a novel accelerated first-order method that achieves the asymptotically optimal convergence rate for smooth functions in the first-order oracle model. 
To this day, Nesterov's Accelerated Gradient Descent (\agd) and variations thereof were the only methods achieving acceleration in this standard blackbox model. 
In contrast, our algorithm is significantly different from \agd, as it relies on a {\it predictor-corrector approach} similar to that used by Mirror-Prox~\cite{Mirror-Prox-Nemirovski} and Extra-Gradient Descent~\cite{extragradient-descent} in the solution of convex-concave saddle point problems. For this reason, we dub our algorithm Accelerated Extra-Gradient Descent (\axgd). 

Its construction is motivated by the discretization of an accelerated continuous-time dynamics~\cite{krichene2015accelerated} using the classical method of implicit Euler discretization. Our analysis explicitly shows the effects of discretization through a conceptually novel primal-dual viewpoint. Moreover, we show that the method is quite general: it attains optimal convergence rates for other classes of objectives (e.g., those with generalized smoothness properties or that are non-smooth and Lipschitz-continuous) using the appropriate choices of step lengths. Finally, we present experiments showing that our algorithm matches the performance of Nesterov's method, while appearing more robust to noise in some cases. 
\end{abstract}

\section{Introduction}\label{sec:intro}

First-order methods for convex optimization have come to play an important role in the design of algorithms and in Theoretical Computer Science in general, with applications including numerical methods~\cite{ST04,KOSZ13}, graph algorithms~\cite{KLOS2014,Sherman2013}, submodular optimization~\cite{Ene} and complexity theory~\cite{JainJiUpadhyayWatrous2009}.

A classical setting for convex optimization is that of smooth optimization, i.e., minimizing a convex differentiable function $f$ over a convex set $X \subseteq \mathbb{R}^n$, with the smoothness assumption that the gradient of $f$ be $L$-Lipschitz continuous\footnote{Lipschitz continuity is defined w.r.t to a pair of dual norms $\|\cdot \|, \|\cdot\|_*$. At a first reading, these can be taken as  $\|\cdot\|_2.$} for some positive real $L,$ i.e.:
$$
\forall \; x,y \in X, \; \|\nabla f(x) - \nabla f(y) \|_* \leq L \cdot \|x-y\|.
$$
In this setting, it is also assumed that the algorithm can access the input function $f$ only via queries to a first-order oracle, i.e., a {\it blackbox}~that on input $x \in X$, returns the vector $\nabla f(x)$ in constant time.\footnote{In general, we may assume that the blackbox also returns the function value $f(x)$. However, for the general class of problems we consider this information is not necessary and the gradient suffices~\cite{nesterov2013introductory}. For intuition about this, see the  expression for the change in duality gap in Equation~\ref{eq:gapdiff}.}

Smooth optimization is of particular interest because it is the simplest setting in which the phenomenon of {\it acceleration} arises, i.e., the optimal algorithms in the blackbox model achieve an error that scales as $O(1/t^2),$ where $t$ is the number of queries~\cite{nesterov1983}. This should be compared to the convergence of steepest-descent methods, which attempt to locally minimize the first-order approximation to the function and only yield $O(1/t)$-convergence~\cite{ben2001lectures,nesterov2013introductory}.
Acceleration has proved an active topic of algorithmic research, both for the promise of obtaining generic speed-ups for problems having some smoothness condition and for the unintuitive nature of the fact that faster algorithms can be obtained by not moving in the direction of steepest-descent.

Recently, a number of papers have helped demystify the concept behind accelerated algorithms by providing interpretations based on continuous dynamics and their discretization~\cite{krichene2015accelerated,wibisono2016variational,su2014differential}, geometric ideas~\cite{Bubeck2015}, and on trading off the performances of two slower first-order methods~\cite{AO-survey-nesterov}.
Despite these efforts,  to this day, Nesterov's Accelerated Gradient Descent (\agd) methods remain the only paradigm~\cite{nesterov1983, Nesterov2004} through which to obtain  accelerated algorithms in the blackbox model and in related settings, where all existing accelerated algorithms are variations of Nesterov's general method~\cite{tseng2008}.

\paragraph*{Our Main Contributions} We present a novel accelerated first-order method that achieves the optimal convergence rate for smooth functions and is significantly different from Nesterov's method, as it relies on a predictor-corrector approach, similar to that of Mirror-Prox~\cite{Mirror-Prox-Nemirovski} and Extra-Gradient Descent~\cite{extragradient-descent}. For this reason, we name our method {\it Accelerated Extra-Gradient Descent}~(\axgd). 
Our derivation of the {\axgd} algorithm is based on the discretization of a recently proposed continuous-time accelerated algorithm~\cite{krichene2015accelerated, wibisono2016variational}. The continuous-time view is particularly helpful in clarifying the relation between \agd, \axgd, and Mirror-Prox. Following~\cite{krichene2015accelerated}, given a gradient field $\nabla f$ and a prox function $\psi$, it is 
possible to define two continuous-time evolutions: the mirror-descent dynamics and the accelerated-mirror-descent dynamics (see Section~\ref{sec:continuous}). With this setup, Nesterov's \agd can be seen as a variant of the classical forward-Euler discretization applied to the accelerated-mirror-descent dynamics. In contrast, Mirror-Prox and extra-gradient methods arise from an approximate backward-Euler discretization~\cite{hairer1993solving} on the mirror-descent dynamics. Finally, our algorithm \axgd is the result of an approximate backward-Euler discretization of the accelerated mirror-descent dynamics.

Another conceptual contribution of our paper is the application of a primal-dual viewpoint on the convergence of first-order methods, both in continuous and discrete time. At every time instant $t$, our algorithm explicitly maintains a current primal solution $\vx^{(t)}$ and a current dual solution $\vz^{(t)}$, the latter in the form of a convex combination of gradients of the convex objective, i.e., a lower-bounding hyperplane.  
This primal-dual pair of solutions yields, for every $t,$  
both an upper bound $U_t$ and a lower bound $L_t$ on the optimum: $U_t \geq f(\vx^*) \geq L_t.$ In all cases, we obtain convergence bounds by explicitly quantifying the rate at which the duality gap $G_t = U_t - L_t$ goes to zero. 
We believe that this primal-dual viewpoint makes the analysis and design of first-order methods easier to carry out. We provide its application to proving other classical results in first-order methods, including Mirror Descent, Mirror-Prox, and Frank-Wolfe algorithms in the upcoming manuscript~\cite{thegaptechnique}.

\paragraph*{Other Technical Contributions} 
In Section~\ref{sec:holder}, we provide a unified convergence proof for standard smooth functions (as defined above) and for functions with H\"older-continuous gradients, a more general notion of smoothness~\cite{nemirovskii1983problem}.
While this paper focuses on the standard smooth setup, the same techniques easily yield results matching those of \agd~methods for the strongly-convex-and-smooth case. Indeed, it is possible to prove that our method is universal, in the sense of Nesterov~\cite{nesterov2015universal}, meaning that it can be composed with a line-search algorithm to yield near-optimal algorithms even when the smoothness parameters of the functions are unknown.
We illustrate this phenomenon by showing that (\axgd) also achieves the optimal rate for the optimization of Lipschitz-continuous convex functions, a non-smooth problem.

Finally, we present a suite of experiments comparing \agd,~\axgd, and standard gradient methods, showing that the performance of \axgd closely matches that of \agd methods. We also explore the empirical performance of \axgd in the practically and theoretically relevant case in which the queried gradients are corrupted by noise. We show that \axgd exhibits better stability properties than \agd in some cases, leading to a number of interesting theoretical questions on the convergence of \axgd.

\subsection{Related Work}

In his seminal work \cite{nesterov1983,Nesterov2004}, Nesterov gave a method for the minimization of convex functions that are smooth with respect to the Euclidean norm, where the function is accessed through a first-order oracle. Nesterov's method converges quadratically faster than gradient descent, at a rate of $O(\frac{1}{t^2})$, which has been shown to be asymptotically optimal~\cite{Nesterov2004} for smooth functions in this standard blackbox model~\cite{nesterov2013introductory}.
More recently, Nesterov generalized this method to allow non-Euclidean norms in the definition of smoothness~\cite{Nesterov2005}. We refer to this generalization of Nesterov's method and to instantiations thereof as \agd methods. 
Accelerated gradient methods have been widely extended and modified for different settings, including composite optimization~\cite{Nesterov2013,Lan2011},  cubic regularization~\cite{nesterov2008cubic}, and universal methods~\cite{nesterov2015universal}. They have also found a number of fundamental applications in many algorithmic areas, including machine learning (see \cite{Bube2014}) and discrete optimization~\cite{LRS2013}. 

An important application of \agd~ methods concerns the solution of various convex-concave saddle point problems. While these are examples of non-smooth problems, for which the optimal rate is known to be $\Omega(\frac{1}{\sqrt{k}})$~\cite{nemirovskii1983problem}, Nesterov showed that the saddle-point structure can be exploited by smoothing the original problem and applying \agd~methods on the resulting smooth function~\cite{Nesterov2005}. This approach~\cite{Nesterov2005, Nesterov2005excessive} yields an $O(\frac{1}{k})$-convergence for convex-concave saddle point problems with smooth gradients. Surprisingly, at around the same time, Nemirovski~\cite{Mirror-Prox-Nemirovski} gave a very different algorithm, known as Mirror-Prox, which  achieves the same complexity for the saddle point problem. Mirror-Prox does not rely on the algorithm or analysis underlying \agd, but is based instead on the idea of an {\it extra-gradient} step, i.e., a correction step that is performed at every iteration to speed up convergence. Mirror-Prox can be viewed as an approximate solution to the 
implicit Euler discretization of the standard mirror descent dynamics of Nemirovski and Yudin~\cite{nemirovskii1983problem}. In this fashion, our \axgd~algorithm resembles Mirror-Prox as it also makes use of an approximate implicit Euler step to discretize a \emph{different}, accelerated dynamic.

A number of interpretations have been proposed to explain the phenomenon of acceleration in first-order methods. Tseng gives a formal framework that unifies all the different instantiations of \agd~methods~\cite{tseng2008}. More recently, Allen-Zhu and Orecchia~\cite{AO-survey-nesterov} cast \agd~methods as the result of coupling mirror descent and gradient descent steps. Bubeck~\etal~give an elegant geometric interpretation of the Euclidean instantiation of Nesterov's method~\cite{Bubeck2015}. At the same time, Su~\etal~\cite{su2014differential}, Krichene~\etal~\cite{krichene2015accelerated}, and Wibisono~\etal~\cite{wibisono2016variational} have provided characterizations of accelerated methods as discretizations of certain families of ODEs related to the gradient flow of the objective $f.$ Our algorithm is strongly influenced by these works: in particular, the starting point for the derivation of \axgd~is the continuous-time accelerated-mirror-descent (\amd) dynamics~\cite{krichene2015accelerated}.

\subsection{Preliminaries}\label{sec:prelims}
We focus on continuous and differentiable functions defined on a closed convex set $X \subseteq \mathbb{R}^n$. We assume that there is an arbitrary (but fixed) norm $\|\cdot\|$ associated with the space, and all the statements about function properties are stated with respect to that norm. We also define the dual norm $\|\cdot\|_*$ in the standard way: $\|\vz\|_* = \sup\{{\innp{\vz, \vx}}: \|\vx\|= 1\}$. The following definitions will be useful in our analysis, and thus we state them here for completeness.  

\begin{definition}\label{def:convexity}
A function $f:X\rightarrow \mathbb{R}$ is convex on $X$, if for all $\vx, \vxh \in X$: $f(\vxh) \geq f(\vx) + \innp{\nabla f(\vx), \vxh - \vx}$.
\end{definition}

\begin{definition}\label{def:smoothness}
A function $f:X\rightarrow \mathbb{R}$ is smooth on $X$ with smoothness parameter $L$ and with respect to a norm $\|\cdot\|$, if for all $\vx, \vxh \in X$: $f(\vxh) \leq f(\vx) + \innp{\nabla f(\vx), \vxh - \vx} + \frac{L}{2}\|\vxh - \vx\|^2$.
\end{definition}

Definition \ref{def:smoothness} can equivalently be stated as: $\|\nabla f(\vx) - \nabla f(\vxh)\|_*\leq L\|\vx - \vxh\|$. 

\begin{definition}\label{def:strong-convexity}
A function $f:X\rightarrow \mathbb{R}$ is strongly convex on $X$ with strong convexity parameter $\sigma$ and with respect to a norm $\|\cdot\|$, if for all $\vx, \vxh \in X$: $f(\vxh) \geq f(\vx) + \innp{\nabla f(\vx), \vxh - \vx} + \frac{\sigma}{2}\|\vxh - \vx\|^2$.
\end{definition}

\begin{definition}\label{def:bregman-divergence}(Bregman Divergence)
$D_{\psi}(\vx, \vxh) \defeq \psi(\vx) - \psi(\vxh)-\innp{\nabla \psi(\vxh), \vx - \vxh}$. 
\end{definition}

\begin{definition}\label{def:cxv-conj}(Convex Conjugate) Function $\psi^*$ is the convex conjugate of $\psi: X \rightarrow \mathbb{R}$, if $\psi^* (\vz) = \max_{\vx \in X}\{\innp{\vz, \vx} - \psi(\vx)\}$, $\forall \vz \in \mathbb{R}$. 
\end{definition}

In the rest of the paper, we will assume that $\psi(\vx)$ is continuously differentiable, so that Fenchel-Moreau Theorem implies that $\psi^{**} = \psi$.\footnote{Note that Fenchel-Moreau Theorem requires $\psi$ to only be lower-semicontinuous for $\psi^{**} = \psi$ to hold, which is a weaker property than continuity or continuous differentiability.} 
We are interested in minimizing a convex function $f$ over $X \subseteq \mathbb{R}^n$. We let $\vx^* = \arg\min_{\vx\in X} f(\vx)$.

We will refer to any step that decreases the value of $f$ as a gradient descent step. In the special case of a smooth function $f$ the gradient descent step from a point $\vx \in X$ will be given as $\mathrm{Grad}(\vx) = \arg\min_{\vxh \in X} \{f(\vx) + \innp{\nabla f(\vx), \vxh - \vx} + \frac{L}{2}\|\vxh-\vx\|^2\}$.

We will assume that there is a strongly-convex differentiable function $\psi:X\rightarrow \mathbb{R}$ such that $\max_{\vx \in X} \{\innp{\vz, \vx} - \psi (\vx)\}$ is easily solvable, possibly in a closed form. Notice that this problem defines the convex conjugate of $\psi(\cdot)$, i.e., $\psi^* (\vz) = \max_{\vx \in X} \{\innp{\vz, \vx} - \psi (\vx)\}$. 
The following standard fact will be extremely useful in carrying out the analysis of the algorithms in this paper. 
\begin{fact}\label{fact:danskin}
Let $\psi: X \to \mathbb{R}$ be a differentiable strongly-convex function. Then:
$$
\nabla \psi^*(\vz) = \arg\max_{\vx \in X} \left\{ \innp{\vz, \vx} - \psi(\vx)\right\}.
$$ 
\end{fact}

Additional useful properties of Bregman divergence are provided in Appendix~\ref{sec:breg-div-prop}.

\section{Accelerated Extra-Gradient Descent}\label{sec:axgd}
In this section, we describe the \axgd~method and analyze its convergence.
For comparison, steps of \agd~and \axgd~are shown next to each other in the box below. In continuous time, both algorithms follow the same  dynamics. However, due to the different discretization methods used in constructing \agd and \axgd, they follow different discrete-time updates. In particular, we show in \cite{thegaptechnique} that \agd can be interpreted as performing explicit (forward) Euler discretization plus a gradient step to correct the discretization error. In contrast, \axgd uses an approximate implementation of implicit (backward) Euler discretization to directly control the discretization error. 

\noindent
\vsepfbox{
\begin{minipage}[t]{.48\textwidth}
\vspace{2pt}
\footnotesize
\textbf{Accelerated Gradient Descent} (\agd)\\
\vspace{12pt}
\begin{equation}\label{eq:dt-acc-grad-desc-step}
\begin{gathered}
\vx^{(k+1)} = \frac{A_k}{A_{k+1}}\vxh^{(k)} + \frac{a_{k+1}}{A_{k+1}}\nabla \psi^*(\vz^{(k)}),\\
\vz^{(k+1)} = \vz^{(k)} - a_{k+1}\nabla f(\vx^{(k+1)}),\\
\vxh^{(k+1)} = \mathrm{Grad}(\vx^{(k+1)}).
\end{gathered}
\end{equation}
\end{minipage} 
\begin{minipage}[t]{.48\textwidth}
\vspace{2pt}
\footnotesize
\textbf{Accelerated Extra-Gradient Descent} (\axgd)
\begin{equation}\label{eq:dt-acc-mp-step}
\begin{gathered}
\vxh^{(k)} = \frac{A_k}{A_{k+1}}\vx^{(k)} + \frac{a_{k+1}}{A_{k+1}}\nabla \psi^*(\vz^{(k)}),\\
\vzh^{(k)} = \vz^{(k)} - a_{k+1}\nabla f(\vxh^{(k)}),\\
\vx^{(k+1)} = \frac{A_k}{A_{k+1}}\vx^{(k)} + \frac{a_{k+1}}{A_{k+1}}\nabla \psi^*(\vzh^{(k)}),\\
\vz^{(k+1)} = \vz^{(k)} - a_{k+1}\nabla f(\vx^{(k+1)}).
\end{gathered}
\end{equation}
\end{minipage}
}

The idea behind \axgd is similar to the dual-averaging version of Nemirovski's mirror prox algorithm \cite{Mirror-Prox-Nemirovski,thegaptechnique}, with the main difference coming from the discretization of the accelerated dynamics in Equation (\ref{eq:ct-amd}) (as opposed to the standard mirror descent dynamics used in \cite{Mirror-Prox-Nemirovski}). As we will show, an exact implicit Euler step would have $\nabla \psi^*(\vz^{(k+1)})$ instead of $\nabla \psi^*(\vzh^{(k)})$ in the third line of \axgd. However, obtaining $\vx^{(k+1)}$ in a such a manner could be computationally prohibitive since $\vz^{(k+1)}$ implicitly depends on $\vx^{(k+1)}$ through its gradient. Instead, we opt for an extra prox-step $\nabla\psi^*(\vzh^{(k)})$ that adds the gradient at an intermediate point $\vxh^{(k)}$ constructed using $\vx^{(k)}$ and $\vz^{(k)}$ from the previous iteration. Thanks to this extra-gradient step, \axgd~can correct the discretization error without using a gradient step.

Convergence proof for \axgd together with the sufficient conditions for obtaining optimal convergence bounds are provided in Section \ref{sec:axgd-conv}. For example, Theorem~\ref{thm:smooth-acc-mp} shows that when the objective function is smooth, \axgd~converges at the optimal rate of $1/k^2$. The analysis of \agd is provided in \cite{thegaptechnique}.

\subsection{Approximate Optimality Gap}\label{sec:agt}

The analysis relies on the construction of an approximate optimality gap $G_t$, which is defined as the difference of an upper bound $U_t$ and a lower bound $L_t$ to the optimal function value $f(\vx^*)$. In particular, for an increasing function of time $t$, $\alpha^{(t)}$, the convergence analysis will work on establishing the following:\\
\vsepfbox{\begin{minipage}[h]{.99\textwidth}\emph{Invariance condition:} $\alpha^{(t)}G_t$ is non-increasing with time $t$.
\end{minipage}}
Such a condition immediately implies: $G_t\leq \frac{\alpha^{(t_0)}}{\alpha^{(t)}}G_{t_0}$, leading to the $\frac{1}{\alpha^{(t)}}$ convergence rate. We sketch the main ideas that relate to the accelerated methods and \axgd in particular here for completeness, while the more general arguments that recover a number of known first-order methods are provided in \cite{thegaptechnique}.

We now describe the upper bound and the lower bound choices, which will take the same form in both continuous time and discrete time domains. To do so, we will rely on the Lebesgue-Stieltjes integration, which allows us to treat continuous and discrete choice of $\alpha^{(t)}$ in a unified manner. Observe that when $\alpha^{(t)}$ is a discrete measure, $\dot{\alpha}^{(t)}$ is a train of (scaled) Dirac Delta functions. Denote $A^{(t)} = \int_{t_0}^t d\alpha^{(\tau)} = \int_{t_0}^t \dot{\alpha}^{(\tau)}d\tau$.

\paragraph*{Upper Bound} As $\vx^*\in X$ is the minimizer of $f(\cdot)$, $f(\vx)$ for any $\vx \in X$ constitutes a valid upper bound. In particular, our choice of the upper bound will be $U_t = f(\vx^{(t)})$, where $\vx^{(t)}$ is the solution maintained by the algorithm at time $t$.

\paragraph*{Lower Bound} More interesting than the upper bound is the construction of a lower bound to $f(\vx^*)$. From convexity of $f$, we have the standard lower-bounding hyperplanes $\forall \vx, \vxh \in X$: $f(\vx) \geq f(\vxh) + \innp{\nabla f(\vxh), \vx - \vxh}$. A natural choice of a lower bound to the optimum at time $t \geq t_0,$ is obtained by averaging such hyperplanes over $[t_0,t]$ according to the measure $\alpha$:
\begin{align} 
f(\vu) \geq  
\frac{ \int_{t_0}^t  f(\vx^{(\tau)}) d\alpha^{(\tau)}}{A^{(t)}} +  \frac{\int_{t_0}^t \innp{\nabla f(\vx^{(\tau)}), \vu - \vx^{(\tau)}} d\alpha^{(\tau)}}{A^{(t)}},\; \forall \vu \in X.\notag
\end{align}
While we could take the minimum over $\vu \in X$ on the right-hand side of this equation as our notion of lower bound, this choice has two serious drawbacks. First, it is non-smooth, and in general not even differentiable, as a function of $t.$  
Second, in continuous-time, it is not defined for our initial time $t_0,$ meaning that we do not have a natural concept of initial lower bound and initial duality gap. (In the discrete time, we can ensure that $\alpha$ contains a Dirac Delta function at $t_0$, which overcomes this issue.) 
We address the first problem by applying regularization, i.e., by adding to both sides of the inequality a regularizer term that is strongly-convex in $\vx$ and then minimizing the right-hand side with respect to $\vu \in X$.\footnote{This is similar to the well-known Moreau-Yosida regularization.}  Without loss of generality, the regularizer can be taken to be the Bregman divergence of a $\sigma$-strongly convex function $\psi$ taken from an input point $\vx^{(t_0).}$ This yields:
\begin{align*}
& f(\vx^*) + \frac{D_{\psi}(\vx^*, \vx^{(t_0)})}{A^{(t)}}\\
&\hspace{.5in} \geq \frac{\int_{t_0}^{t} f(\vx^{(\tau)}) d\alpha^{(\tau)}}{A^{(t)}} +  \frac{\min_{\vu \in X}\left\{\int_{t_0}^{t} \innp{\nabla f(\vx^{(\tau)}), \vu - \vx^{(\tau)}}d\alpha^{(\tau)} + D_{\psi}(\vu, \vx^{(t_0)}) \right\}}{A^{(t)}}.
\end{align*}
To address the second problem, we mix into the $\alpha$-combination of hyperplanes the optimal 
lower bound $f(\vx^*)$ with weight $\alpha^{(t)} - A^{(t)}$ (which is just zero in the discrete time, as in that case $A^{(t)}=\alpha^{(t)}$). Rescaling the normalization factor, we obtain our notion of {\it regularized lower bound}:
\begin{equation} \label{eq:lower}
\begin{aligned}
L_{t} \defeq &\frac{\int_{t_0}^{t} f(\vx^{(\tau)}) d\alpha^{(\tau)}}{\alpha^{(t)}} +  \frac{\min_{\vu \in X}\left\{\int_{t_0}^{t} \innp{\nabla f(\vx^{(\tau)}), \vu - \vx^{(\tau)}} d\alpha^{(\tau)} + D_{\psi}(\vu, \vx^{(t_0)}) \right\}}{\alpha^{(t)}}\\
&+ \frac{(\alpha^{(t)}-A^{(t)})f(\vx^*) - D_{\psi}(\vx^*, \vx^{(t_0)})}{\alpha^{(t)}}.
\end{aligned}
\end{equation}

\subsection{Accelerated Mirror Descent in Continuous Time}\label{sec:continuous}

We now show that the accelerated dynamics can be obtained by enforcing the invariance condition from previous subsection with $\alpha^{(t)}G_t$ being constant; i.e., we enforce that $\frac{d}{dt}(\alpha^{(t)}G_t)=0$. Towards that goal, assume that $\alpha^{(t)}$ is continuously differentiable, and observe that $\alpha^{(t)}-A^{(t)} = \alpha^{(t_0)}$ is constant. To simplify the notation when taking the time derivative of $\alpha^{(t)}G^{(t)}$, we first show the following:
\begin{proposition}\label{prop:arg-min}
Let $\vz^{(t)} = \nabla \psi(\vx^{(t_0)}) - \int_{t_0}^t \nabla f(\vx^{(\tau)})d\alpha^{(\tau)}$. Then:
$$
\nabla \psi^*(\vz^{(t)}) = \arg\min_{\vu \in X}\left\{\int_{t_0}^{t} \innp{\nabla f(\vx^{(\tau)}), \vu - \vx^{(\tau)}} d\alpha^{(\tau)} + D_{\psi}(\vu, \vx^{(t_0)}) \right\}.
$$
\end{proposition}
I.e., $\nabla\psi^*(\vz^{(t)})$ is the argument of the minimum appearing in the definition of lower bound $L_t$. The proof is simple and is provided in the appendix.

Recalling that $U_t = f(\vx^{(t)})$ and using (\ref{eq:lower}) and Danskin's theorem (which allows us to differentiate inside the min):
\begin{align} \label{eq:gapdiff}
\frac{d}{dt}(\alpha^{(t)}G_t) & = \frac{d}{dt}(\alpha^{(t)}f(\vx^{(t)})) - \dot{\alpha}^{(t)}f(\vx^{(t)}) - \dot{\alpha}^{(t)}\innp{\nabla f(\vx^{(t)}), \nabla \psi^*(\vz^{(t)})-\vx^{(t)}} \nonumber  \\
& = \innp{\nabla f(\vx^{(t)}), {\alpha}^{(t)} \vx^{(t)} -\dot{\alpha}^{(t)}\left(\nabla \psi^*(\vz^{(t)})-\vx^{(t)}\right)}.
\end{align}
Hence, to obtain $\frac{d}{dt}(\alpha^{(t)}G_t) = 0$, it suffices to set $\alpha^{(t)}\vx^{(t)} = \dot{\alpha}^{(t)}(\nabla \psi^*(\vz^{(t)})-\vx^{(t)})$, resulting in the accelerated dynamics from \cite{krichene2015accelerated}:
\begin{equation}\label{eq:ct-amd}
\begin{gathered}
\dot{\vz}^{(t)} = - \dot{\alpha}^{(t)}\nabla f(\vx^{(t)}),\\
\dot{\vx}^{(t)} = \dot{\alpha}^{(t)}\frac{\nabla \psi^*(\vz^{(t)})-\vx^{(t)}}{\alpha^{(t)}},\\
\vz^{(t_0)} = \nabla\psi(\vx^{(t_0)}),\; \vx^{(t_0)}\in X \text{ is an arbitrary initial point.}
\end{gathered}
\end{equation}

It is not hard to see that (\ref{eq:ct-amd}) constructs a sequence of points $\vx^{(t)}$ that are feasible, that is, $\vx^{(t)}\in X$. This is because $\vx^{(t)}$ can equivalently be written as $\frac{d}{dt}(\alpha^{(t)}\vx^{(t)}) = \dot{\alpha}^{(t)}\nabla \psi^*(\vz^{(t)})$, which, after integrating over $\tau \in [t_0, t]$, gives $\vx^{(t)} = \frac{\alpha^{(t_0)}}{\alpha^{(t)}}\vx^{(t_0)} + \frac{1}{\alpha^{(t)}}\int_{t_0}^t \nabla \psi^*(\vz^{(\tau)})d\alpha^{(\tau)}$ -- a convex combination of $\vx^{(t_0)}$ and $\nabla \psi^*(\vz^{(\tau)})$ for $\tau \in [t_0, t]$. By (\ref{eq:ct-amd}), $\vx^{(t_0)}\in X$, while $\nabla\psi^*(\vz^{(\tau)})\in X$ by Proposition \ref{prop:arg-min}.

We immediately obtain the following continuous-time convergence guarantee:
\begin{lemma}\label{lemma:ct-amd-conv}
Let $\vx^{(t)}$ evolve according to (\ref{eq:ct-amd}). Then, $\forall t \geq t_0$:
$$
f(\vx^{(t)}) - f(\vx^*) \leq \frac{\alpha^{(t_0)}(f(\vx^{(t_0)})-f(\vx^*))+D_{\psi}(\vx^*, \vx^{(t_0)})}{\alpha^{(t)}}.
$$
\end{lemma}
\begin{proof}
We have already established that $\frac{d}{dt}(\alpha^{(t)}G^{(t)})=0$, and, therefore, $f(\vx^{(t)})-f(\vx^*)\leq G_t = \frac{\alpha^{(t_0)}}{\alpha^{(t)}}G_{t_0}$. Observing that $G_{t_0} = f(\vx^{(t_0)}) - f(\vx^*) + D_{\psi}(\vx^*, \vx^{(t_0)})/\alpha^{(t_0)}$, the proof follows.
\end{proof}

\subsection{Discretization}

As discussed in Section \ref{sec:agt}, our construction of the approximate optimality gap is valid both in the continuous time and in the discrete time domain. To understand where the discretization error occurs, we make the following observations. First, the upper bound does not involve any integration, and thus cannot incur a discretization error. In the lower bound (\ref{eq:lower}), the role of the first integral is only to perform weighted averaging, which is the same in the continuous time and in the discrete time, and, therefore, does not incur a discretization error. The terms that are not integrated over look the same whether or not $\alpha^{(t)}$ is discrete. Therefore, the only term that can incur the discretization error is the integral under the min: $I^{(t_0, t)} = \int_{t_0}^t \innp{\nabla f(\vx^{(\tau)}), \nabla \psi^*(\vz^{(t)})-\vx^{(\tau)}}d\alpha^{(\tau)}$.   

As mentioned before, when $\alpha$ is a discrete measure, we can express it as $\alpha^{(t)} = \sum_{i=1}^{\infty} a_i \delta(t-(t_0+i-1))$, where $\delta(\cdot)$ denotes the Dirac Delta function and $a_i$'s are positive. Then $A^{(t)} = \int_{t_0}^t d\alpha^{(\tau)}= \sum_{i: t_0+ i-1 \leq t} a_i$. 
To simplify the notation, we will use $i \in \mathbb{Z}_+$ to denote the discrete time points corresponding to $t_0 + i -1$ on the continuous line. Therefore, the discretization error incurred in $A^{(t)}L_t$ between the discrete time points $i$ and $i+1$ (understood as integrating from $i^+$ to $(i+1)^+$) is $I^{(i, i+1)}-I_c^{(i, i+1)}$, where $I_c^{(i, i+1)}$ is the continuous approximation of $I^{(i, i+1)}$ (i.e., we allow continuous integration rules in $I_c^{(i, i+1)}$). We can now establish the following bound on the discretization error. 

\begin{lemma}\label{lemma:discr-error}
Let $A_{i+1}G_{i+1}-A_iG_i \equiv E_{i+1}$ be the discretization error. Then $$G_{k} = \frac{A_1}{A_k}G_1 + \frac{\sum_{i=1}^k E_i}{A_k}$$ and $$E_{i+1} \leq \innp{\nabla f(\vx^{(i+1)}), A^{(i+1)}\vx^{(i+1)}-A^{(i)}\vx^{(i)}-a_{i+1}\nabla\psi^*(\vz^{(i+1)})} - D_{\psi^*}(\vz^{(i)}, \vz^{(i+1)}).$$
\end{lemma}
\begin{proof}
The first part of the lemma follows by summing over $1\leq i\leq k$. For the second part, we have already argued that $E_{i+1} = I_c^{(i, i+1)}-I^{(i, i+1)}$. For the discrete integral $I^{(i, i+1)}$, as $\dot{\alpha}^{(t)}$ just samples the function under the integral at point $i+1$, we have:
\begin{equation}\label{eq:I-i-i+1}
I^{(i, i+1)} = a_{i+1}\innp{\nabla f(\vx^{(i+1)}), \nabla \psi^*(\vz^{(i+1)})-\vx^{(i+1)}}.
\end{equation}
For the continuous integral, using (\ref{eq:ct-amd}) and integration by parts:
\begin{align}
I_c^{(i, i+1)} &= \int_{i}^{i+1}\alpha^{(\tau)}\innp{\nabla f(\vx^{(\tau)}), \dot{\vx}^{(\tau)}}d\tau + \int_{i}^{i+1}\innp{\nabla f(\vx^{(\tau)}), \nabla \psi^*(\vz^{(i+1)})-\nabla\psi^*(\vz^{(\tau)})}d\alpha^{(\tau)}\notag\\
&= A^{(i)}(f(\vx^{(i+1)})-f(\vx^{(i)})) - \int_{i}^{i+1}\innp{\dot{\vz}^{(\tau)}, \nabla \psi^*(\vz^{(i+1)})-\nabla\psi^*(\vz^{(\tau)})}d\tau\notag\\
&= A^{(i)}(f(\vx^{(i+1)})-f(\vx^{(i)})) - D_{\psi^*}(\vz^{(i)}, \vz^{(i+1)}), \label{eq:I_c-i-i+1}
\end{align}
where we have used $\dot{\vz}^{(\tau)}=-\dot{\alpha}^{(\tau)}\nabla f(\vx^{(\tau)})$,  $\nabla_{\vz^{(\tau)}} D_{\psi^*}(\vz^{(\tau)}, \vz^{(i+1)}) = \nabla\psi^*(\vz^{(\tau)})-\nabla\psi^*(\vz^{(i+1)})$, and $D_{\psi^*}(\vz^{(i)}, \vz^{(i)})=0$.

By convexity of $f$, $f(\vx^{(i+1)})-f(\vx^{(i)})\leq \innp{\nabla f(\vx^{(i+1)}), \vx^{(i+1)}-\vx^{(i)}}$. Combining with (\ref{eq:I-i-i+1}) and (\ref{eq:I_c-i-i+1}):
$$
E_{i+1} \leq \innp{\nabla f(\vx^{(i+1)}), A^{(i+1)}\vx^{(i+1)}-A^{(i)}\vx^{(i)}-a_{i+1}\nabla\psi^*(\vz^{(i+1)})} - D_{\psi^*}(\vz^{(i)}, \vz^{(i+1)}),
$$
as claimed.
\end{proof}
We remark that the same result for the discretization error can be obtained by directly computing $A_{i+1}G_{i+1}-A_{i}G_{i}$ under a discrete measure $\alpha$ (where all the integrals in the definition of the duality gap are replaced by summations). We have chosen to work with the integration error described above to demonstrate the cause of the discretization error.

We now describe how \axgd~cancels out the discretization error by (approximately) implementing implicit Euler discretization of $\dot{\vx}^{(t)}$.

\paragraph*{Implicit Euler Discretization} Implicit Euler discretization is an abstract discretization method which defines the next iterate $\vx^{(k+1)}$ implicitly as a function of the gradient at $\vx^{(k+1)}$. In the case of the \amd~dynamics, implicit Euler discretization yields the following algorithm: let $\vx^{(1)}\in X$ be an arbitrary initial point that satisfies $\vx^{(1)} = \nabla \psi^*(\vz^{(1)})$, where $\vz^{(1)} = \nabla \psi(\vx^{(1)}) - \nabla f(\vx^{(1)})$; for all $k \geq 1$
\begin{align}\label{eq:implicit-euler-step}
\begin{cases}
\vz^{(k+1)} = \vz^{(k)} - a_{k+1} \nabla f(\vx^{(k+1)}),\\
\vx^{(k+1)} = \frac{A_k}{A_{k+1}}\vx^{(k)} + \frac{a_{k+1}}{A_{k+1}}\nabla \psi^*(\vz^{(k+1)})
\end{cases}
\end{align}
Observe that $\vx^{(k+1)}$ in (\ref{eq:implicit-euler-step}) exactly sets the inner product in $E_{i+1}$ (Lemma \ref{lemma:discr-error}) to zero, leaving only the negative term $-D_{\psi^*}(\vz^{(i)}, \vz^{(i+1)})$. 
While this discretization is not computationally feasible in practice, as it requires solving for the implicitly defined $\vx^{(k+1)},$ it also boasts a negative discretization error, i.e., it converges faster than the continuous-time \amd. Ultimately, we will use this extra slack to trade-off the error arising from an {\em approximate} implicit discretization.

\subsection{Convergence of AXGD}\label{sec:axgd-conv}

A standard way to implement implicit Euler discretization in the solution of ODEs~\cite{hairer1993solving} is to replace the exact solution of the implicit equation with a small number of fixed point iterations of the same equation. In our case, the implicit equation can be written as:
\begin{align}
\vx^{(k+1)} = \frac{A_k}{A_{k+1}}\vx^{(k)}+\frac{a_{k+1}}{A_{k+1}}\nabla\psi^*(\vz^{(k)}- a_{k+1} \nabla f(\vx^{(k+1)})).\notag 
\end{align}
Two steps of the fixed-point iteration yield the following updates, which are exactly those performed by \axgd:
\begin{align}
\begin{cases}
\vxh^{(k)} = \frac{A_k}{A_{k+1}}\vx^{(k)}+\frac{a_{k+1}}{A_{k+1}}\nabla\psi^*(\vz^{(k)}).\notag \\
\vx^{(k+1)} = \frac{A_k}{A_{k+1}}\vx^{(k)}+\frac{a_{k+1}}{A_{k+1}}\nabla\psi^*(\vz^{(k)}- a_{k+1} \nabla f(\vxh^{(k)}))
\end{cases}
\end{align}

We can now analyze \axgd~as producing an approximate solution to the implicit Euler discretization problem. The following lemma gives a general bound on the convergence of \axgd~for a convex and differentiable $f(\cdot)$ without additional assumptions. The only (mild) difference is replacing $D_{\psi}(\vx, \vx^{(1)})$ and $D_{\psi}(\vx^*, \vx^{(1)})$ by $D_{\psi}(\vx, \vxh^{(0)})$ and $D_{\psi}(\vx^*, \vxh^{(0)})$, since we start from the ``intermediate'' point $\vxh^{(0)}$. This change is only important for bounding the initial gap $G_1$; everything else is the same as before.

\begin{lemma}\label{lemma:acc-mirror-prox} %
Consider the \axgd~ algorithm as described in Equation~(\ref{eq:dt-acc-mp-step}), starting from an arbitrary point $\vxh^{(0)}$ with $\vz^{(0)} = \nabla \psi(\vxh^{(0)})$ and $A_0 = 0$. Then the error from Lemma~\ref{lemma:discr-error} is bounded by:
\begin{align*}
E_{i+1} \leq & a_{i+1}\innp{\nabla f(\vx^{(i+1)}) - \nabla f(\vxh^{(i)}), \nabla \psi^*(\vzh^{(i)})- \nabla \psi^*(\vz^{(i+1)})}\\
&- D_{\psi^*}(\vzh^{(i)}, \vz^{(i+1)}) - D_{\psi^*}(\vz^{(i)}, \vzh^{(i)}).
\end{align*}

\end{lemma}
\begin{proof}
From Lemma \ref{lemma:discr-error}: 
\begin{align}
E_{i+1} \leq &  a_{i+1} \innp{\nabla f(\vx^{(i+1)}),\nabla \psi^*(\vzh^{(i)}) - \nabla \psi^*(\vz^{(i+1)})} - D_{\psi^*}(\vz^{(i)}, \vz^{(i+1)}) \notag\\
 = & a_{i+1} \innp{\nabla f(\vx^{(i+1)}) - \nabla f(\vxh^{(i)}) + \nabla f(\vxh^{(i)}),\nabla \psi^*(\vzh^{(i)}) - \nabla \psi^*(\vz^{(i+1)})}\notag\\
 &- D_{\psi^*}(\vz^{(i)}, \vz^{(i+1)}). \notag 
\end{align}
We now use the fact that $a_{i+1}f(\vxh^{(i)}) = \vz^{(i)} - \vzh^{(i)}$ together with the standard triangle-inequality for Bregman divergences (see Proposition~\ref{prop:magic-identity}) to show that:
\begin{align}
a_{i+1}\innp{\nabla f(\vxh^{(i)}), \nabla \psi^*(\vzh^{(i)}) - \nabla \psi^*(\vz^{(i+1)})} = \innp{\vz^{(i)} - \vzh^{(i)},\nabla \psi^*(\vzh^{(i)}) - \nabla \psi^*(\vz^{(i+1)})} \notag \\
= D_{\psi^*}(\vz^{(i)}, \vz^{(i+1)}) - D_{\psi^*}(\vzh^{(i)}, \vz^{(i+1)}) - D_{\psi^*}(\vz^{(i)}, \vzh^{(i)}),\notag
\end{align}
Combining the results of the last two equations, we get the claimed bound on the error.
\end{proof}

%
%
%
\subsection{Smooth Minimization with AXGD}

We show that \axgd~achieves the asymptotically optimal convergence rate of $1/k^2$ for the minimization of an $L$-smooth convex objective $f(\cdot)$ by applying  Lemma \ref{lemma:acc-mirror-prox}. The crux of the proof is that we can take sufficiently large steps while keeping the error from Lemma~\ref{lemma:acc-mirror-prox} non-positive. In other words, we are able to move quickly through the continuous evolution of \amd~by taking large discrete steps.

\begin{theorem}\label{thm:smooth-acc-mp}
Let $f:X\rightarrow \mathbb{R}$ be an $L$-smooth convex function and let $\vx^{(k)}, \vz^{(k)}, \vxh^{(k)}, \vzh^{(k)}$ be updated according to the \axgd~algorithm in Equation~(\ref{eq:dt-acc-mp-step}), starting from an arbitrary initial point $\vxh^{(0)}\in X$ with the following initial conditions: $\vz^{(0)} = \nabla\psi(\vxh^{(0)})$ and $A_0 = 0$. Let $\psi:X\rightarrow \mathbb{R}$ be $\sigma$-strongly convex. If %
$\frac{{a_{k}}^2}{A_k} \leq \frac{\sigma}{L},$
then for all $k \geq 1$, 
$$
f(\vx^{(k)}) - f(\vx^*) \leq \frac{D_{\psi}(\vx^*, \vxh^{(0)})}{A_k}.
$$
In particular, if $a_k = \frac{k+1}{2}\cdot \frac{\sigma}{L}$ and $\psi(\vx) = \frac{\sigma}{2}\|\vx\|^2$, then:
$$
f(\vx^{(k)})-f(\vx^*) \leq \frac{2L}{(k+1)^2}\|\vx^* - \vxh^{(0)}\|^2.
$$
\end{theorem}
\begin{proof}
The proof follows directly by applying Lemma \ref{lemma:acc-mirror-prox} and using $L$-smoothness of $f$ and $\sigma$-strong convexity of $\psi$. In particular, by Cauchy-Schwartz inequality and smoothness: 
\begin{align}
&\innp{\nabla f(\vx^{(k+1)}) - \nabla f(\vxh^{(k)}), \nabla \psi^*(\vzh^{(k)}) - \nabla \psi^*(\vz^{(k+1)})}\notag\\
&\hspace{1in}\leq L \|\vx^{(k+1)} - \vxh^{(k)}\|\cdot\|\nabla \psi^*(\vz^{(k+1)}) - \nabla \psi^*(\vzh^{(k)})\|,\notag
\end{align}
and, by Proposition \ref{prop:cvx-conj-bd-is-strongly-cvx-too}
\begin{equation}
\begin{aligned}
& D_{\psi^*}(\vzh^{(k)}, \vz^{(k+1)}) + D_{\psi^*}(\vz^{(k)}, \vzh^{(k)})\\
& \hspace{.5in}\geq \frac{\sigma}{2}\left(\|\nabla \psi^*(\vzh^{(k)}) - \nabla \psi^*(\vz^{(k+1)})\|^2 + \| \nabla \psi^*(\vz^{(k)}) - \nabla \psi^*(\vzh^{(k)}) \|^2\right).
\end{aligned}
\end{equation}
From the definition of the steps, $\vx^{(k+1)} - \vx^{(k)} = \frac{a_{k+1}}{A_{k+1}}(\nabla \psi^*(\vzh^{(k)}) - \nabla \psi^*(\vz^{(k)}))$, and, therefore:
\begin{align*}
E_{k+1} \leq \frac{{a_{k+1}}^2}{A_{k+1}}L\cdot pq - \frac{\sigma}{2}(p^2 + q^2),
\end{align*}
where $p = \|\nabla \psi^*(\vzh^{(k)}) - \nabla \psi^*(\vz^{(k+1)})\|$ and $q = \| \nabla \psi^*(\vz^{(k)}) - \nabla \psi^*(\vzh^{(k)}) \|$. Since, for any $p, q$, $p^2 + q^2 - 2\alpha pq \geq 0$ whenever $\alpha \leq 1$, it follows that $E_{k+1}\leq 0$ whenever $\frac{{a_{k+1}}^2}{A_{k+1}}\frac{L}{\sigma} \leq 1$, which is true by the theorem assumptions. In particular, for $a_k = \frac{k+1}{2}\cdot \frac{\sigma}{L}$, $A_k = \frac{\sigma}{L}(\frac{(k+1)(k+2)}{4}) \geq \frac{\sigma}{L}\frac{(k+1)^2}{4}$. 
This proves that $f(\vx^{(k)}) - f(\vx^*) \leq \frac{G_1}{A_{k}}.$ It remains to bound $G_1.$ This a simple computation, shown in the appendix, which yields: $G_1 \leq 
\frac{1}{A_1} D_{\psi}(\vx^*, \vxh^{(0)}).$
\end{proof}

\subsection{Generalized Smoothness: H\"{o}lder-Continuous Gradients}\label{sec:holder}

Suppose that $f(\cdot)$ has H\"{o}lder-continuous gradients, namely, $f(\cdot)$ then satisfies:
\begin{equation}\label{eq:holder-continuous-grad}
\|\nabla f(\vxh) - \nabla f(\vx)\| \leq L_{\nu}\|\vxh - \vx\|^{\nu},
\end{equation}
which also implies:
\begin{equation}\label{eq:holder-continuous-grad-1}
\forall \vx, \vxh \in X:\; f(\vxh) \leq f(\vx) + \innp{\nabla f(\vx), \vxh - \vx} + \frac{L_{\nu}}{1+\nu} \|\vxh - \vx\|^{1+\nu},
\end{equation}
where $\nu \in (0, 1]$, $L_{\nu} \in \mathbb{R}_{++}$. In particular, if $\nu = 1$, then $f(\cdot)$ is $L_{\nu}$-smooth. Thus, the functions with H\"{o}lder-continuous gradients represent a class of functions with generalized/relaxed smoothness properties.

The lower iteration complexity bound for (unconstrained) minimization of convex functions with  H\"{o}lder-continuous gradients was established in~\cite{nemirovskii1983problem} and equals $O\left({L_{\nu}D_1^{1+\nu}}{\epsilon^{-\frac{2}{1+3\nu}}}\right)$, where $D_1$ is the distance from the initial point to the optimal solution. A matching upper bound was obtained in \cite{nemirovski1985optimal}.

To recover the optimal convergence rate in the minimization of convex functions with H\"{o}lder-continuous gradients, as before, we bound the discretization error from Lemma \ref{lemma:acc-mirror-prox}. Before doing so, we will need the following technical proposition (which appears in a similar form as Lemma 3.1 a) in \cite{Mirror-Prox-Nemirovski}).

\begin{proposition}\label{prop:ineq-needed-for-mp-holder}
$$a_{i+1}\innp{\nabla f(\vx^{(i+1)}) - \nabla f(\vxh^{(i)}), \nabla \psi^*(\vz^{(i+1)}) - \nabla \psi^*(\vzh^{(i)})} \leq \sigma^{-1}{a_{i+1}}^2 \|\nabla f(\vx^{(i+1)}) - \nabla f(\vxh^{(i)})\|^2.$$
\end{proposition}
The proof is provided in the appendix.

\begin{theorem}\label{eq:acc-mp-for-holder-smooth-fns}
Let $f(\cdot)$ be a convex function that satisfies (\ref{eq:holder-continuous-grad}), and let $\psi(\cdot)$ be $\sigma$-strongly convex. Let $\vx^{(k)}, \vz^{(k)}, \vxh^{(k)}, \vzh^{(k)}$ be updated according to the \axgd~algorithm in Equation~(\ref{eq:dt-acc-mp-step}), starting from an arbitrary initial point $\vxh^{(0)}\in X$ with the following initial conditions: $\vz^{(0)} = \nabla\psi(\vxh^{(0)})$ and $A_0 = 0$. Let $a_{k} = c\frac{\sigma}{L_{\nu}}D^{1-\nu}k^{\frac{-1+3\nu}{2}}$, where $D = \max_{\vx, \vxh \in X}\|\vx - \vxh\|$ and $c = 2^{\frac{3\nu(\nu + 1)-1}{2}}$. Then, $\forall k \geq 1:$
$$
f(\vx^{(k)}) - f(\vx^*) \leq 2^{\frac{1-3\nu(\nu+1)}{2}}\frac{L_{\nu}}{\sigma} \frac{ D^{\nu-1} D_{\psi}(\vx^{*}, \vxh^{(0)})}{k^{\frac{1+3\nu}{2}}}.
$$
In particular, if $\psi(\vx) = \frac{\sigma}{2}\|\vx\|^2$, then:
\end{theorem}
$$
f(\vx^{(k)}) - f(\vx^*) \leq 2^{\frac{1-3\nu(\nu+1)}{2}} L_{\nu} D^{1+\nu}k^{-\frac{1+3\nu}{2}}.
$$
\begin{proof}
We prove the theorem by bounding the discretization error $E_{i+1}$ from Lemma \ref{lemma:acc-mirror-prox}. 
Applying Propositions \ref{prop:cvx-conj-bd-is-strongly-cvx-too} and \ref{prop:ineq-needed-for-mp-holder}:
\begin{align} E_{i+1} = &
{a_{i+1}}\innp{\nabla f(\vx^{(i+1)}) - \nabla f(\vxh^{(i)}), \nabla \psi^*(\vzh^{(i)})- \nabla \psi^*(\vz^{(i+1)})}\notag\\
&- D_{\psi^*}(\vzh^{(i)}, \vz^{(i+1)})-D_{\psi^*}(\vz^{(i)}, \vzh^{(i)})\notag\\
\leq  &\sigma^{-1}{a_{i+1}}^2 \|\nabla f(\vx^{(i+1)}) - \nabla f(\vxh^{(i)})\|^2\notag\\
&- \frac{\sigma}{2}\left(\|\nabla \psi^*(\vzh^{(i)}) - \nabla \psi^*(\vz^{(i+1)})\|^2 + \|\nabla\psi^*(\vz^{(i)}) - \nabla\psi^*(\vzh^{(i)})\|^2\right)\notag\\
\leq &\sigma^{-1}{a_{i+1}}^2{L_{\nu}}^2 \|\vx^{(i+1)} - \vxh^{(i)}\|^{2\nu} - \frac{\sigma}{2}\|\nabla\psi^*(\vz^{(i)}) - \nabla\psi^*(\vzh^{(i)})\|^2\notag\\
\leq &\sigma^{-1}{L_{\nu}}^2 \frac{{a_{i+1}}^{2 + 2\nu}}{{A_{i+1}}^{2\nu}}\|\nabla\psi^*(\vz^{(i)}) - \nabla\psi^*(\vzh^{(i)})\|^{2\nu} - \frac{\sigma}{2} \|\nabla\psi^*(\vz^{(i)}) - \nabla\psi^*(\vzh^{(i)})\|^2, \label{eq:holder-acc-mp-cond}
\end{align}
where the second inequality is by (\ref{eq:holder-continuous-grad}) and the third inequality is by the step definition (\ref{eq:dt-acc-mp-step}).

Taking $a_{k} = c\frac{\sigma}{L_{\nu}}D^{1-\nu}k^{\frac{-1+3\nu}{2}}$, where $c = 2^{\frac{3\nu(\nu + 1)-1}{2}}$, it follows that $A_k = \sum_{i=1}^k a_i \geq \sum_{i=\lceil k/2\rceil}^k a_i \geq \frac{c}{2} D^{1-\nu}\frac{\sigma}{L_{\nu}} \big(\frac{k}{2}\big)^{\frac{1+3\nu}{2}}$. Therefore, the expression in (\ref{eq:holder-acc-mp-cond}) is at least:
\begin{align*}
\Big(-c^2 2^{3\nu(\nu+1)}(k+1)^{\nu - 1} + \frac{1}{2}\Big)\sigma\|\nabla\psi^*(\vz^{(k)}) - \nabla\psi^*(\vzh^{(k)})\|^2 \geq 0,
\end{align*}
as $(k+1)^{\nu - 1}\leq 1$. Therefore, we have that $G_k \leq \frac{A^{(1)}}{A^{(k)}}G_1$, and using similar arguments to bound the initial gap $G_1$, the proof follows.
\end{proof}

\subsection{Non-Smooth Minimization: Lipschitz-Continuous Objective}

We now show that we can recover the well-known $\frac{1}{\sqrt{k}}$ convergence rate for the class of non-smooth $L$-Lipschitz objectives by using \axgd. This is summarized in the following theorem. We note that, as in the analysis of classical mirror descent (see, e.g., \cite{ben2001lectures}), the factor $\log(k)$ can be removed if we fix the approximation error (and, consequently, the number of steps $k$) in advance. 
\begin{theorem}\label{thm:lipschitz-cont}
Let $f(\cdot)$ be a Lipschitz-continuous function with parameter $L$. Let $\vx^{(k)}$, $\vz^{(k)}$, $\vxh^{(k)}$, $\vzh^{(k)}$ be updated according to the \axgd~algorithm in Equation~(\ref{eq:dt-acc-mp-step}), starting from an arbitrary initial point $\vxh^{(0)}\in X$ with the following initial conditions: $\vz^{(0)} = \nabla\psi(\vxh^{(0)})$ and $A_0 = 0$. If $a_k = \frac{\sqrt{\sigma}}{2\sqrt{2}L}\sqrt{\frac{D_{\psi}(\vx^*, \vxh^{(0)})}{k}}$, then, $\forall k\geq 1$:
$$
f(\vx^{(k)}) - f(\vx^*)\leq 8(2+\log(k))\frac{L \cdot \sqrt{D_{\psi}(\vx^*, \vxh^{(0)})}}{\sqrt{\sigma}\sqrt{k}}.
$$
In particular, for $\psi(\vx) = \frac{\sigma}{2} \|\vx\|^2$:
$$
f(\vx^{(k)}) - f(\vx^*)\leq 4\sqrt{2}(2+\log(k))\frac{L \cdot \|\vx^* - \vxh^{(0)}\|}{\sqrt{k}}.
$$
\end{theorem}
\begin{proof}
As before, we bound the discretization error from Lemma \ref{lemma:acc-mirror-prox}. 
As $f(\cdot)$ is $L$-Lipschitz, using Proposition \ref{prop:cvx-conj-bd-is-strongly-cvx-too}:
\begin{align*}
E_{i+1} \leq & 
{a_{i+1}}\innp{\nabla f(\vx^{(i+1)}) - \nabla f(\vxh^{(i)}),  \nabla \psi^*(\vzh^{(i)}) - \nabla \psi^*(\vz^{(i+1)})} \\
&
- D_{\psi^*}(\vzh^{(i)}, \vz^{(i+1)})-D_{\psi^*}(\vz^{(i)}, \vzh^{(i)})\\
\leq & 2a_{i+1}L\|\nabla \psi^*(\vz^{(i+1)}) - \nabla \psi^*(\vzh^{(i)})\| - \frac{\sigma}{2}\|\nabla \psi^*(\vz^{(i+1)}) - \nabla \psi^*(\vzh^{(i)})\|^2\\
\leq & \frac{8(a^{(i+1)}L)^2}{\sigma},
\end{align*}
where we have used the inequality $2xy - x^2 \leq y^2$, $\forall x, y$. As $\sigma \geq L$ and $$
A_{k}\cdot \frac{2\sqrt{2}L}{\sqrt{\sigma D_{\psi}(\vx^*, \vxh^{(0)})}} = \sum_{i=1}^k \frac{1}{\sqrt{k}}\geq \sum_{i=\lceil k/2\rceil}^k \frac{1}{\sqrt{k}}\geq \frac{1}{2}\cdot \sqrt{\frac{k}{2}},$$ 
we have that 
$$
\sum_{i=1}^k \frac{E_i}{A_k} \leq 8 \cdot \frac{L \cdot \sqrt{{D_{\psi}(\vx^*, \vxh^{(0)})}}}{\sqrt{\sigma}\sqrt{k}} (\log(k)+1),$$ 
which, after bounding the initial gap by similar arguments, completes the proof. 
\end{proof}

%
%
%
\section{Experiments}\label{sec:experiments}

We now illustrate the performance of \agd and \axgd for (i) an unconstrained problem over $\mathbb{R}^n$ with the objective function $f(\vx) = \frac{1}{2}\innp{\mA \vx, \vx} - \innp{\vb, \vx}$, and (ii) for the problem with the same objective and unit simplex as the feasible region, where $\mA$ is the Laplacian of a cycle graph\footnote{Namely, the sum of a tridiagonal matrix $\mathbf{B}$ with 2's on its main diagonal and -1's on its remaining two diagonals and a matrix $\mathbf{C}$ whose all elements are zero except for the $\mathbf{C}_{1, n} = \mathbf{C}_{n, 1} = -1$.} and $\vb$ is a vector whose first element is one and the remaining elements are zero. This example is known as a ``hard'' instance for smooth minimization -- it is typically used in proving the lower iteration complexity bound for first-order methods (see, e.g., \cite{nesterov2013introductory}). We also include Gradient Descent (\gd)  in the exact gap graphs for comparison. 
In the experiments, we take $n=100$ and $\sigma=L$ ($=4$). We use the $\ell_2$ norm in the gradient steps. 

\begin{figure}[t!]
\centering
\subfigure[unconstrained]{\label{fig:ex-gap}\includegraphics[width=.24\textwidth]{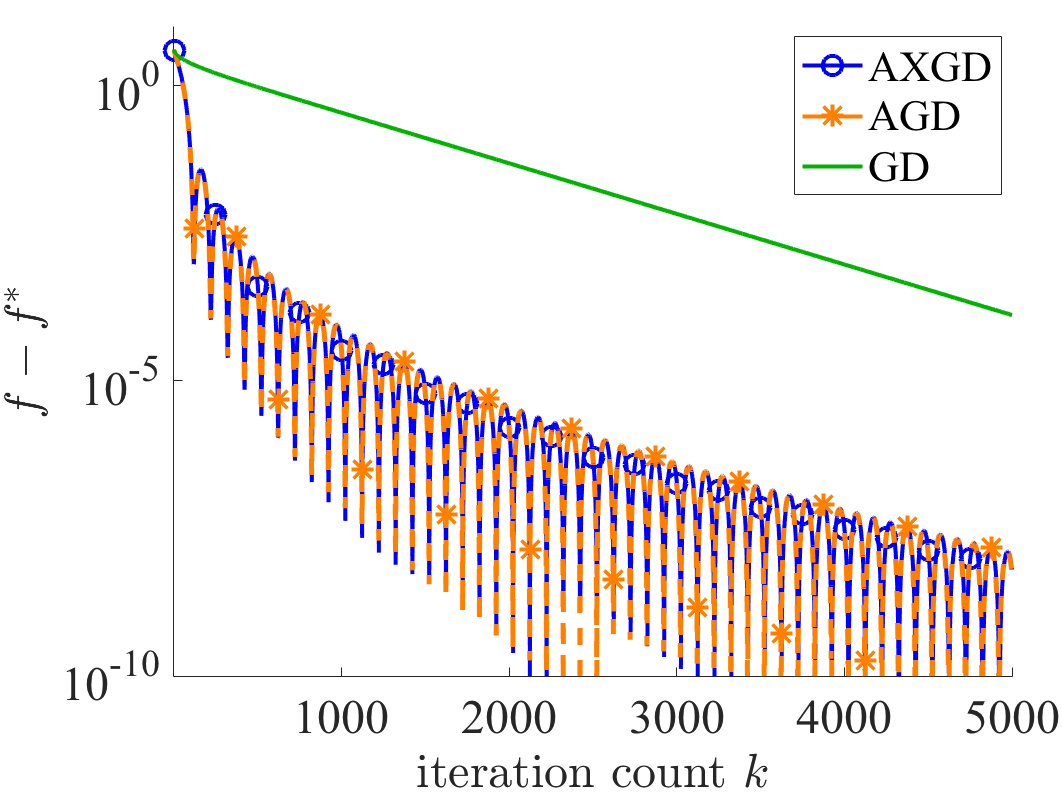}}
\subfigure[unconstrained]{\label{fig:ex-approx-gap}\includegraphics[width=.24\textwidth]{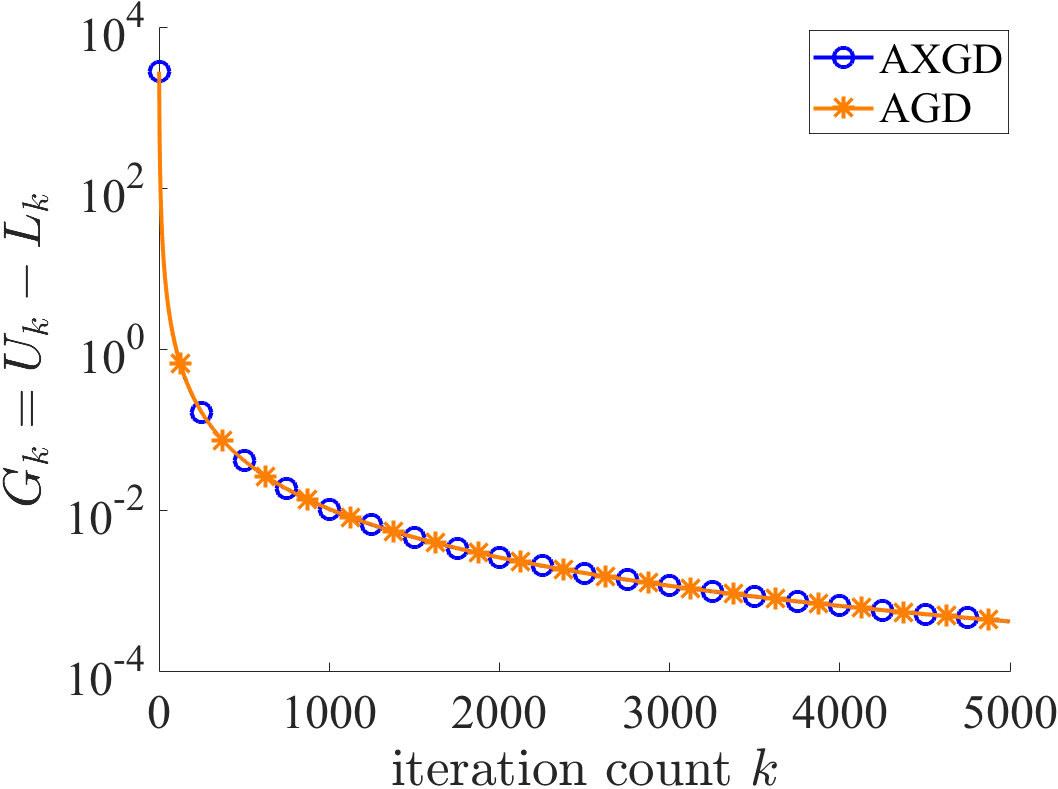}}
\subfigure[simplex]{\label{fig:ex-gap-simplex}\includegraphics[width=.24\textwidth]{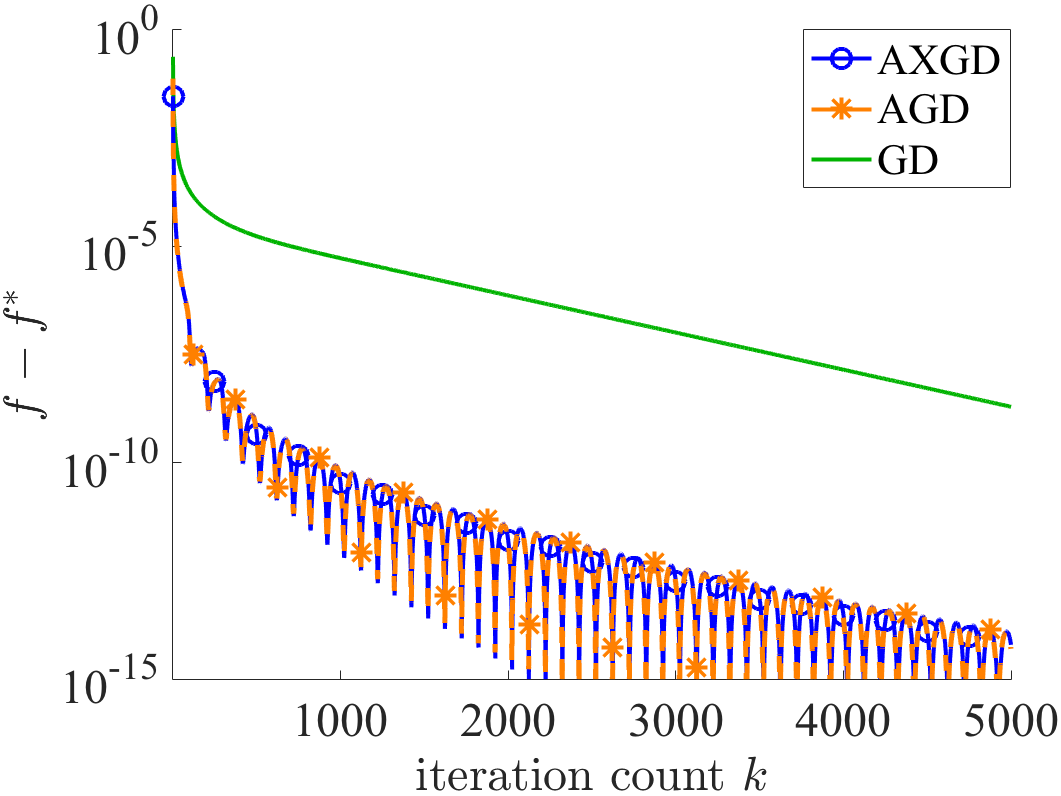}}
\subfigure[simplex]{\label{fig:ex-approx-gap-simplex}\includegraphics[width=.24\textwidth]{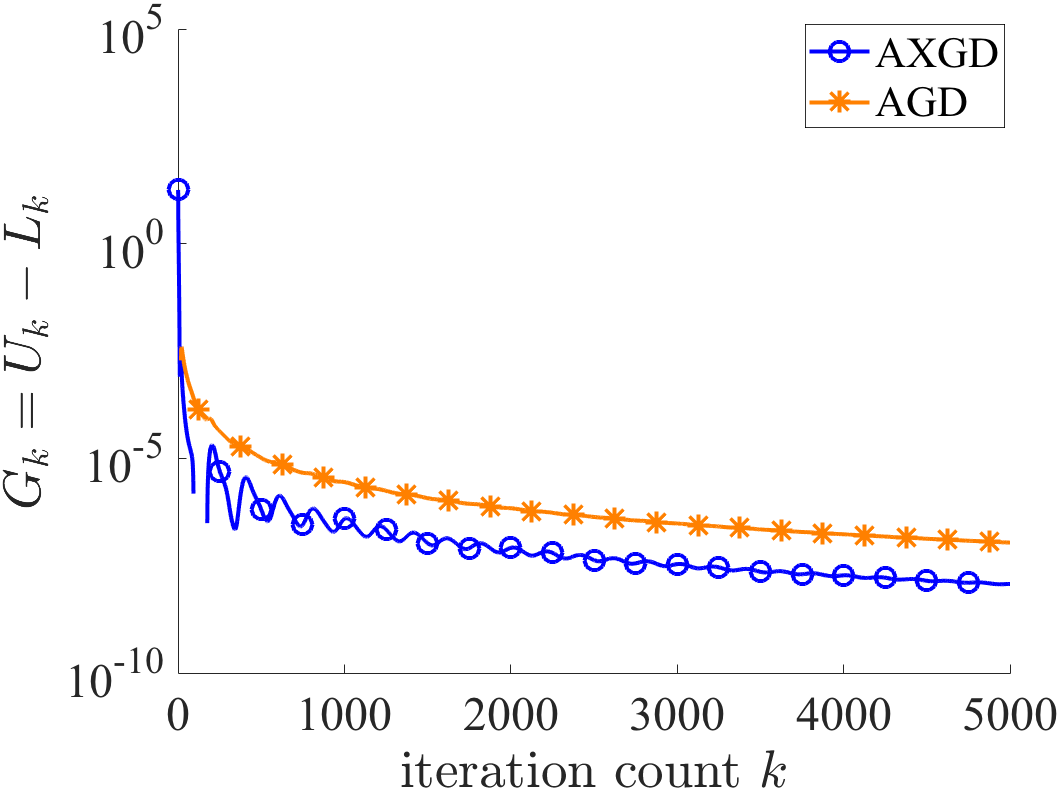}}
\caption{(\protect\subref{fig:ex-gap}),(\protect\subref{fig:ex-gap-simplex}) Exact and (\protect\subref{fig:ex-approx-gap}),(\protect\subref{fig:ex-approx-gap-simplex}) approximate duality gaps for \agd and \axgd with exact gradients.}
\label{fig:gaps-exact}
\end{figure}

In the figures, $f$ denotes the objective value at the upper-bound point and $f^*$ denotes the optimal objective value, so that $f- f^*$ is the true distance to the optimum (the exact gap). 
Fig.~\ref{fig:gaps-exact} shows the distance to the optimum and the approximate  duality gap $G_k = U_k - L_k$ obtained using our analysis. We can observe that \agd~and \axgd~exhibit similar performance in these examples. The approximate gap overestimates the actual duality gap, however, the difference between the two decreases with the number of iterations.

\begin{figure}[t]
\centering
\subfigure[$\epsilon_{\eta} = 10^{-1}$, unconstrained]{\label{fig:noisy-gap-e-1}\includegraphics[width=.3\textwidth]{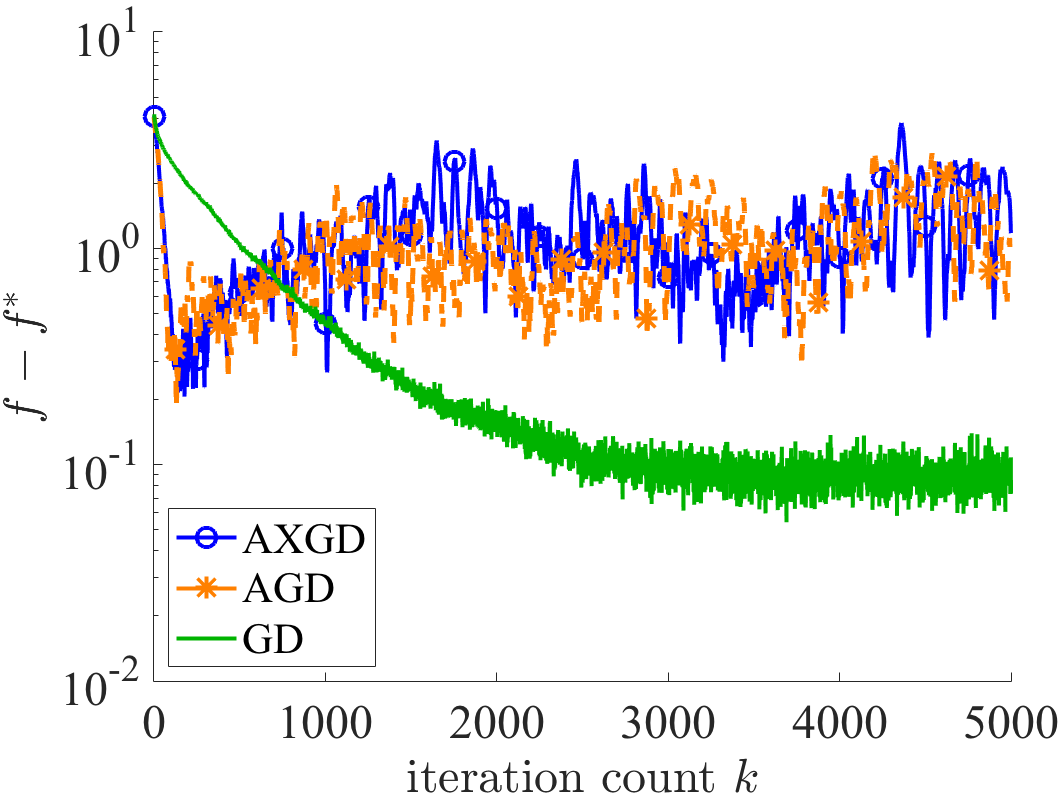}}
\subfigure[$\epsilon_{\eta} = 10^{-2}$, unconstrained]{\label{fig:noisy-gap-e-2}\includegraphics[width=.3\textwidth]{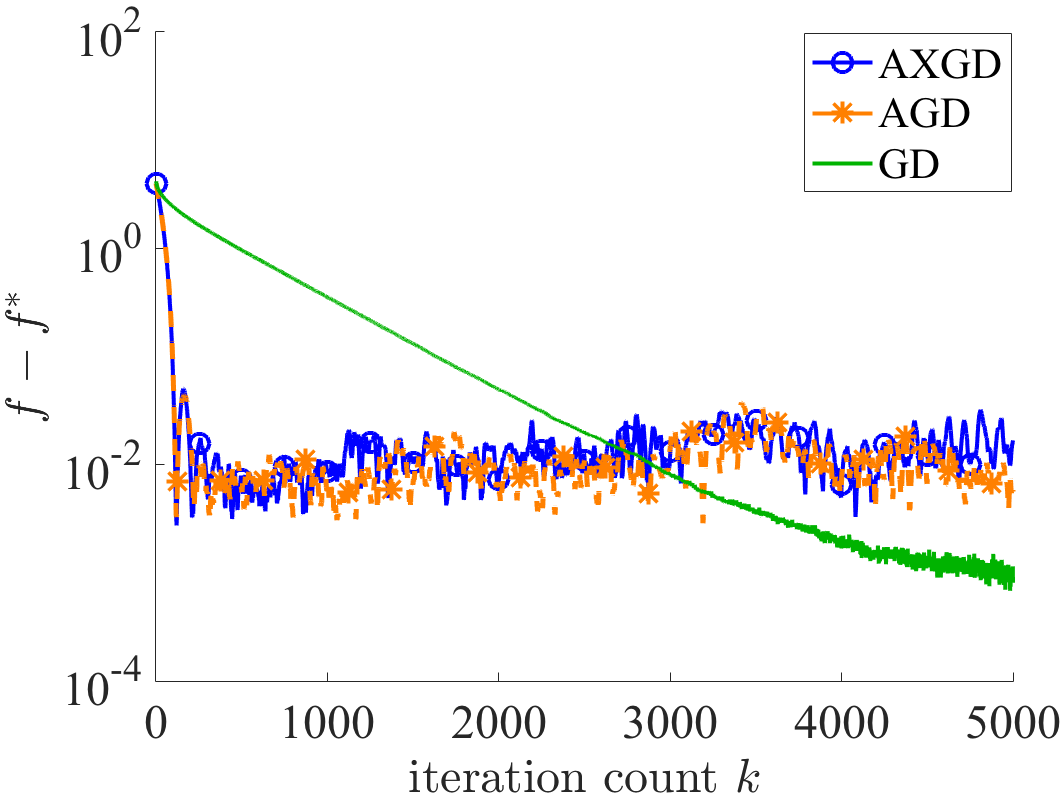}}
\subfigure[$\epsilon_{\eta} = 10^{-3}$, unconstrained]{\label{fig:noisy-gap-e-3}\includegraphics[width=.3\textwidth]{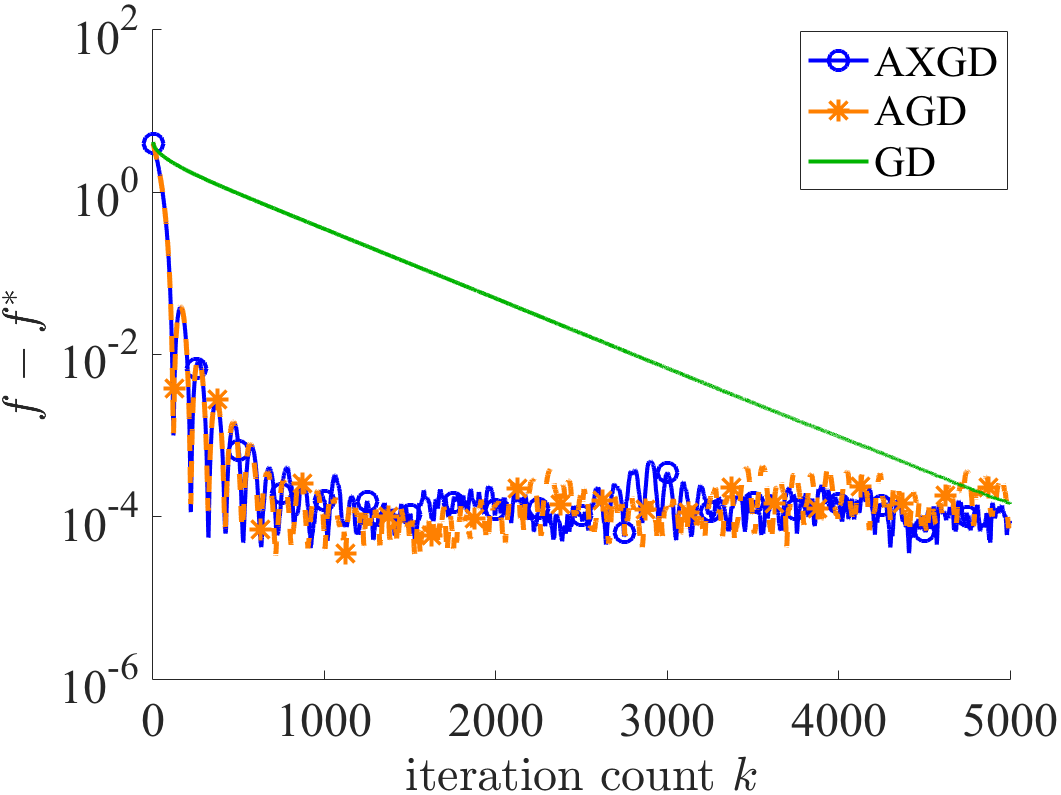}}\\
\subfigure[$\epsilon_{\eta} = 10^{-1}$, simplex]{\label{fig:noisy-gap-e-1-simplex}\includegraphics[width=.3\textwidth]{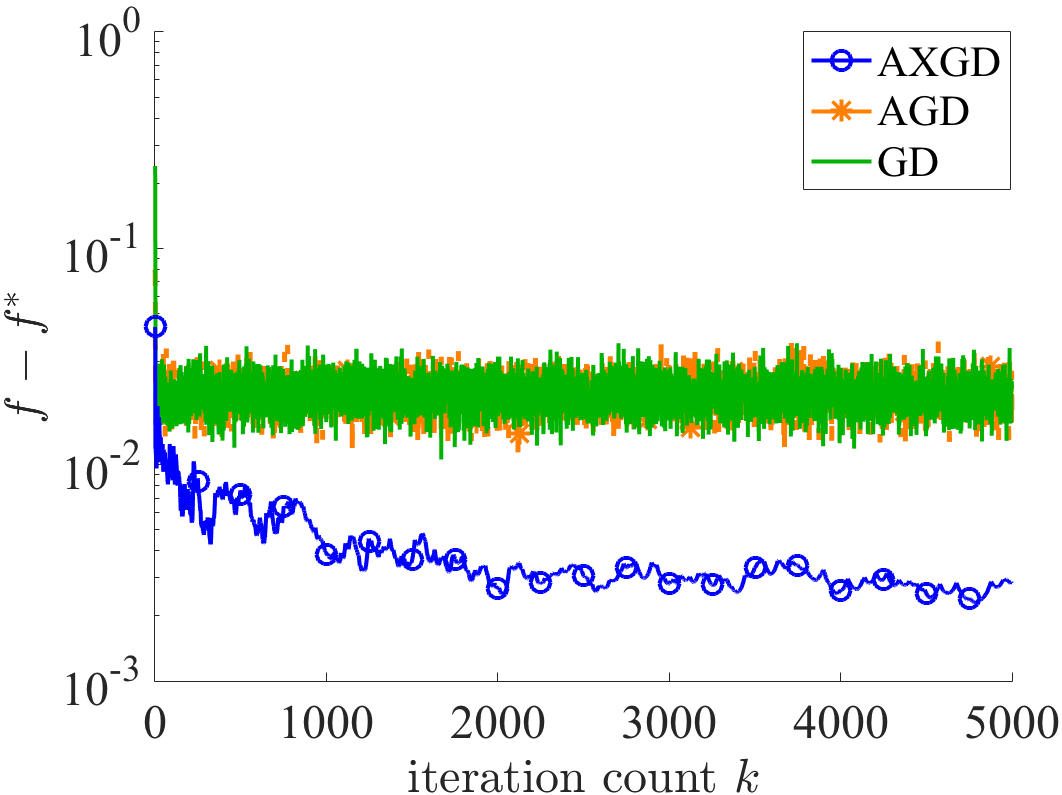}}
\subfigure[$\epsilon_{\eta} = 10^{-2}$, simplex]{\label{fig:noisy-gap-e-2-simplex}\includegraphics[width=.3\textwidth]{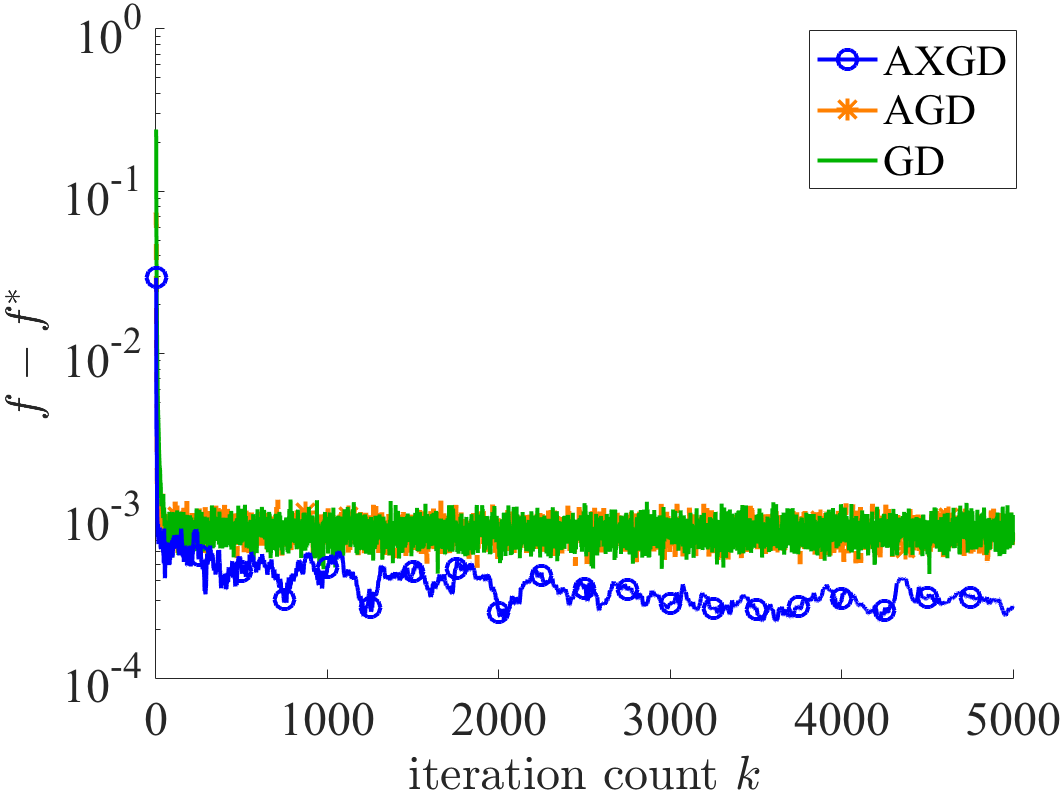}}
\subfigure[$\epsilon_{\eta} = 10^{-3}$, simplex]{\label{fig:noisy-gap-e-3-simplex}\includegraphics[width=.3\textwidth]{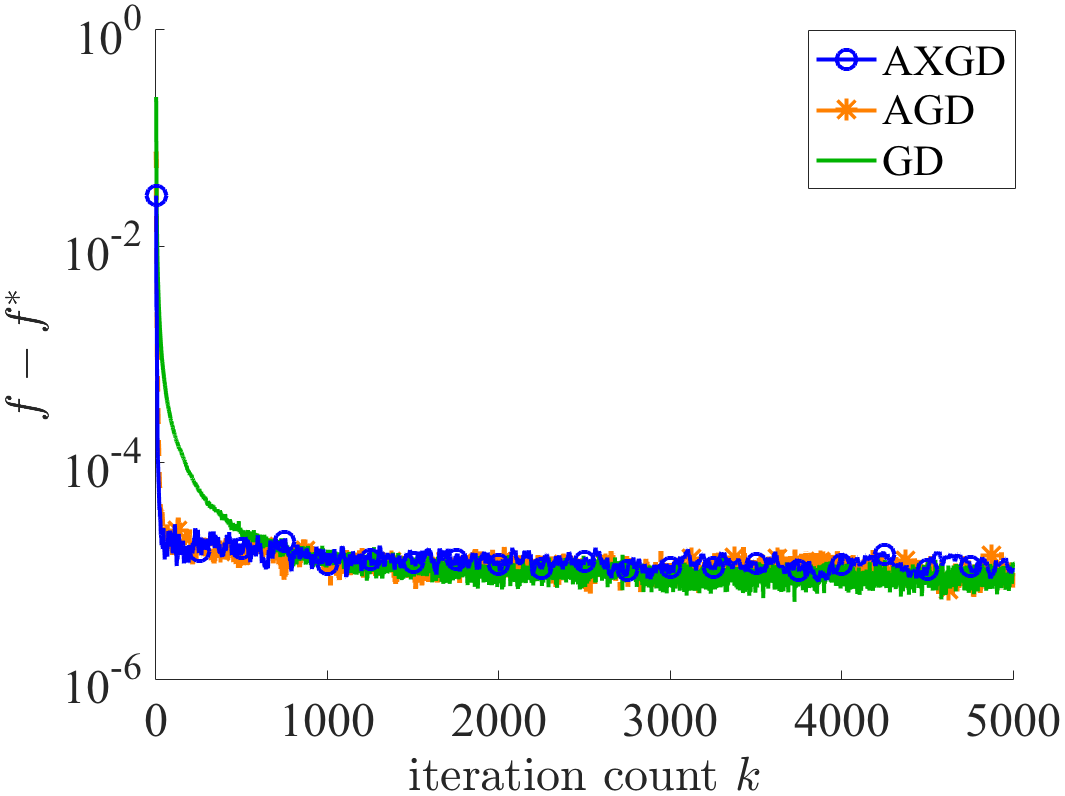}}
\caption{Exact gap for additive Gaussian noise in the gradients with zero mean and covariance $\epsilon_{\eta} I$ (\protect\subref{fig:noisy-gap-e-1})-(\protect\subref{fig:noisy-gap-e-3}) in the unconstrained-region case and (\protect\subref{fig:noisy-gap-e-1-simplex})-(\protect\subref{fig:noisy-gap-e-3-simplex}) in unit simplex.}
\label{fig:noisy-gaps}
\end{figure}

\paragraph*{Acceleration and Noise} 
We now consider the setting in which the gradients output by our oracle are corrupted by additive noise, which has significant applications in practice~\cite{moritz} and theory~\cite{smith}. We note that this model is fundamentally different from the inexact model considered by Devolder~\etal~\cite{devolder2014first}, for which tight lower bounds preventing acceleration exist.\footnote{In \cite{devolder2014first}, it is assumed that a function $f(\cdot)$ is associated with a $(\delta, L)$ oracle, such that 
$
f(\vxh)\leq f(\vx) + \innp{\nabla f(\vx), \vxh - \vx} + \frac{L}{2}\|\vxh - \vx\|^2 + \frac{\delta}{2},\; \forall \vx, \vxh \in X.
$
Such a model seems more suitable for incorrectly specified functions (e.g., non-smooth functions treated as being smooth) and adversarially perturbed functions.}

Specifically, we experimentally evaluate the performance of \agd~and \axgd~under additive Gaussian noise. 
Fig.~\ref{fig:noisy-gaps} illustrates the performance of \agd~and \axgd~when the gradients are corrupted by zero-mean additive Gaussian noise with covariance matrix $\epsilon_{\eta} I$, where $I$ is the identity matrix. When the region is unconstrained (top row in Fig.~\ref{fig:noisy-gaps}), both \agd and \axgd exhibit high sensitivity to noise.  
The \gd method overall exhibits higher tolerance to noise (at the expense of slower convergence). 
In the case of the unit simplex region (bottom row in Fig.~\ref{fig:noisy-gaps}), all the algorithms appear more tolerant to noise than in the unconstrained case. Interestingly, on this example \axgd exhibits higher tolerance to noise than \gd and \agd, both in terms of mean and in terms of variance. Explaining this phenomenon analytically is an interesting question that merits further investigation. 

\section{Conclusion}
We have presented a novel accelerated method -- \axgd-- that combines ideas from the Nesterov's \agd~and Nemirovski's mirror prox. \axgd~achieves optimal convergence rates for a range of convex optimization problems, such as the problems with the (i) smooth objectives, (ii) objectives with H\"{o}lder-continuous gradients, (iii) and non-smooth Lipschitz-continuous objectives. In the constrained-regime  experiments from Section~\ref{sec:experiments}, the method demonstrates favorable performance compared to \agd when subjected to zero-mean Gaussian noise. 

There are several directions that merit further investigation. A more thorough analytical and experimental study of acceleration when the gradients are corrupted by noise is of particular interest, since the gradients can often come from noise-corrupted measurements. Further, our experiments from Fig.~\ref{fig:noisy-gaps} suggest that there are cases that incur a trade-off between noise tolerance and acceleration. A systematic study of this trade-off is thus another important direction, since it would guide the choice of accelerated/non-accelerated algorithms in practice depending on the application. Finally, it is interesting to investigate whether restart schemes can improve  the algorithms' noise tolerance, since in the noiseless setting several restart schemes are known to improve the convergence of \agd in practice.

\bibliographystyle{plainurl}
\bibliography{references}

\appendix

\section{Properties of the Bregman  Divergence}\label{sec:breg-div-prop}

The following properties of Bregman divergence will be useful in our analysis.
\begin{proposition}\label{prop:Bregman-divergence-properties}
$D_{\psi}(\nabla \psi^{*}(\vz), \vx) = D_{\psi^*}(\nabla \psi(\vx), \vz)$, $\forall \vx, \vz$.
\end{proposition}
\begin{proof}
From the definition of $\psi^*$ and Fact \ref{fact:danskin}, 
\begin{equation}\label{eq:psi*-in-terms-of-psi}
\psi^*(\vz) = \innp{\nabla \psi^*(\vz), \vz} - \psi(\nabla \psi^*),\; \forall \vz.
\end{equation}
Similarly, as in the light of Fenchel-Moreau Theorem $\psi^{**} = \psi$,
\begin{equation}\label{eq:psi-in-terms-of-psi*}
\psi(\vx) = \innp{\nabla \psi(\vx), \vx} - \psi^*(\nabla \psi(\vx)),\; \forall \vx.
\end{equation}
Using the definition of $D_{\psi}(\nabla \psi^*(\vz), \vx)$ and Fact \ref{fact:danskin}:
\begin{align}
D_{\psi}(\nabla \psi^*(\vz), \vx) &= \psi(\nabla \psi^*(\vz)) - \psi(\vx) - \innp{\nabla \psi (\vx), \nabla \psi^*(\vz) - \vx}\notag\\
&= \psi(\nabla \psi^*(\vz)) + \psi^*(\nabla \psi(\vx)) - \innp{\nabla \psi(\vx), \nabla \psi^*(\vz)}. \label{eq:D-psi-nabla-psi*-z-x}
\end{align}
Similarly, using the definition of $D_{\psi^*}(\nabla \psi(\vx), \vz)$ combined with (\ref{eq:psi*-in-terms-of-psi}):
\begin{align}
D_{\psi^*}(\nabla \psi(\vx), \vz)&= \psi^*(\nabla \psi(\vx)) - \psi^*(\vz) - \innp{\nabla \psi^*(\vz), \nabla \psi(\vx) - \vz}\notag\\
&= \psi^*(\nabla \psi(\vx)) + \psi(\nabla \psi^*(\vz)) - \innp{\nabla \psi^*(\vz), \nabla \psi(\vx)}.\label{eq:D-psi*-nabla-psi-x-z}
\end{align}
Comparing (\ref{eq:D-psi-nabla-psi*-z-x}) and (\ref{eq:D-psi*-nabla-psi-x-z}), the proof follows.
\end{proof}

\begin{proposition}\label{prop:cvx-conj-bd-is-strongly-cvx-too}
If $\psi(\cdot)$ is $\sigma$-strongly convex, then $D_{\psi^*}(\vz, \vzh) \geq \frac{\sigma}{2}\|\nabla \psi^*(\vz) - \nabla \psi^*(\vzh)\|^2$.
\end{proposition}
\begin{proof}
Using the definition of $D_{\psi^*}(\vz, \vzh)$ and (\ref{eq:psi*-in-terms-of-psi}), we can write $D_{\psi^*}(\vz, \vzh)$ as:
$$
D_{\psi^*}(\vz, \vzh) = \psi(\nabla \psi^*(\vzh)) - \psi(\nabla\psi^*(\vz)) - \innp{\vz, \nabla \psi^*(\vzh) - \nabla \psi^*(\vz)}.
$$
Since $\psi(\cdot)$ is $\sigma$-strongly convex, it follows that:
$$
D_{\psi^*}(\vz, \vzh) \geq \frac{\sigma}{2}\|\nabla \psi^*(\vzh) - \nabla \psi^*(\vz)\|^2 + \innp{\nabla \psi(\nabla \psi^*(\vz))-\vz, \nabla \psi^*(\vzh) - \nabla \psi^*(\vz)}.
$$
As, from Fact \ref{fact:danskin}, $\nabla\psi^*(\vz) = \arg\max_{\vx \in X}\{\innp{\vx, \vz} - \psi(\vx)\}$, by the first-order optimality condition $$\innp{\nabla \psi(\nabla \psi^*(\vz))-\vz, \nabla \psi^*(\vzh) - \nabla \psi^*(\vz)} \geq 0,$$ completing the proof.
\end{proof}

The Bregman divergence $D_{\psi^*}(\vx,\vy)$ captures the difference between $\psi^*(\vx)$ and its first order approximation at $\vy.$ Notice that, for a differentiable $\psi^*$, we have:
$$
\nabla_{\vx} D_{\psi^*}(\vx,\vy) = \nabla \psi^*(\vx) - \nabla \psi^*(\vy).
$$
The Bregman divergence $D_{\psi^*}(\vx,\vy)$ is a convex function of $\vx.$ Its Bregman divergence is itself.
\begin{proposition}\label{prop:magic-identity}
For all $\vx, \vy, \vz \in X$ 
$$
D_{\psi^*}(\vx, \vy) = D_{\psi^*}(\vz, \vy) + \innp{\nabla \psi^*	(\vz) - \nabla \psi^*(\vy), \vx - \vz} + D_{\psi^*}(\vx, \vz).
$$
\end{proposition}


\section{Omitted Proofs from Section \ref{sec:axgd}}

\begin{repproposition}{prop:arg-min}
Let $\vz^{(t)} = \nabla \psi(\vx^{(t_0)}) - \int_{t_0}^t \nabla f(\vx^{(\tau)})d\alpha^{(\tau)}$. Then:
$$
\nabla \psi^*(\vz^{(t)}) = \arg\min_{\vu \in X}\left\{\int_{t_0}^{t} \innp{\nabla f(\vx^{(\tau)}), \vu - \vx^{(\tau)}} d\alpha^{(\tau)} + D_{\psi}(\vu, \vx^{(t_0)}) \right\}.
$$
\end{repproposition}
\begin{proof}
From the definition of Bregman divergence:
\begin{align*}
&\arg \min_{\vu \in X}\left\{\int_{t_0}^{t} \innp{\nabla f(\vx^{(\tau)}), \vu - \vx^{(\tau)}} d\alpha^{(\tau)} + D_{\psi}(\vu, \vx^{(t_0)}) \right\}\\
&\;= \arg\min_{\vu \in X}\left\{\int_{t_0}^{t} \innp{\nabla f(\vx^{(\tau)}), \vu - \vx^{(\tau)}} d\alpha^{(\tau)} + \psi(\vu) - \psi(\vx^{(t_0)})-\innp{\nabla\psi(\vx^{(t_0)}),\vu - \vx^{(t_0)}}\right\}\\
&\;= \arg\min_{\vu \in X}\left\{\innp{\int_{t_0}^{t}\nabla f(\vx^{(\tau)})d\alpha^{(\tau)} - \nabla\psi(\vx^{(t_0)}), \vu} + \psi(\vu)\right\}.
\end{align*}
Using the definition of $\vz^{(t)}$ and Fact \ref{fact:danskin}, the proof follows.
\end{proof}

\begin{proof}[Remaining Proof of Theorem \ref{thm:smooth-acc-mp} (The Bound on $G_1$)]
To bound $G_1$, we recall the definition of $L_1$:
\begin{align}
L_1 &= f(\vx^{(1)}) + \min_{\vx\in X}\left\{ \innp{\nabla f(\vx^{(1)}), \vx - \vx^{(1)}} + \frac{1}{A_1}D_{\psi}(\vx, \vxh^{(0)}) \right\} - \frac{1}{A_1}D_{\psi}(\vx^*, \vxh^{(0)})\notag\\
&= f(\vx^{(1)}) + \innp{\nabla f(\vx^{(1)}), \nabla\psi^*(\vz^{(1)}) - \vx^{(1)}} + \frac{1}{A_1}D_{\psi}(\nabla\psi^*(\vz^{(1)}), \vxh^{(0)}) - \frac{1}{A_1} D_{\psi}(\vx^*, \vxh^{(0)})\notag.
\end{align}
As $a_1 = A_1$, $\vx^{(1)}=\nabla\psi^*(\vzh^{(0)})$, and $a_1\nabla f(\vxh^{(0)}) = \vz^{(0)}-\vzh^{(0)}$, using Proposition \ref{prop:magic-identity}, we have that:
\begin{align}
\innp{\nabla f(\vxh^{(0)}), \nabla\psi^*(\vz^{(1)}) - \vx^{(1)}} &= \frac{1}{A_1}\innp{\vz^{(0)}-\vzh^{(0)}, \nabla\psi^*(\vz^{(1)})-\nabla\psi^*(\vzh^{(0)})}\notag\\
&= \frac{1}{A_1}\left( D_{\psi^*}(\vz^{(0)}, \vzh^{(0)}) - D_{\psi^*}(\vz^{(0)}, \vz^{(1)}) + D_{\psi^*}(\vzh^{(0)}, \vz^{(1)}) \right).\label{eq:dt-l1-acc-mp}
\end{align}
On the other hand, by smoothness of $f(\cdot)$ and the initial condition: 
\begin{align}\label{eq:dt-l1-acc-mp-2}
\innp{\nabla f(\vx^{(1)})-\nabla f(\vxh^{(0)}), \nabla\psi^*(\vz^{(1)}) - \vx^{(1)}}\geq -L\|\nabla \psi^*(\vzh^{(0)}) - \vxh^{(0)}\|\|\nabla\psi^*(\vz^{(1)}) - \vx^{(1)}\|.
\end{align}
Finally, by Proposition \ref{prop:Bregman-divergence-properties} and the initial condition $\vz^{(0)}=\nabla\psi(\vxh^{(0)})$, we have that $D_{\psi^*}(\vz^{(0)}, \vz^{(1)})=D_{\psi}(\nabla \psi^*(\vz^{(1)}), \vxh^{(0)})$. Combining with (\ref{eq:dt-l1-acc-mp}), (\ref{eq:dt-l1-acc-mp-2}), and $G_1 = U_1 - L_1 = f(\vx^{(1)})-L_1$:
\begin{align}
G_1 \leq & L\|\nabla \psi^*(\vzh^{(0)}) - \vxh^{(0)}\|\cdot\|\nabla\psi^*(\vz^{(1)}) - \vx^{(1)}\|\notag\\
& - \frac{1}{A_1}\left(  D_{\psi^*}(\vz^{(0)}, \vzh^{(0)})+ D_{\psi^*}(\vzh^{(0)}, \vz^{(1)}) \right) + \frac{1}{A_1} D_{\psi}(\vx^*, \vxh^{(0)})\notag\\
=& L\|\nabla \psi^*(\vzh^{(0)}) - \vxh^{(0)}\|\cdot\|\nabla\psi^*(\vz^{(1)}) - \vx^{(1)}\|\notag\\
&- \frac{1}{A_1}\left(  D_{\psi}(\nabla\psi^*(\vzh^{(0)}), \vxh^{(0)}) + D_{\psi^*}(\vzh^{(0)}, \vz^{(1)}) \right) + \frac{1}{A_1} D_{\psi}(\vx^*, \vxh^{(0)})  \notag\\
\leq & L\|\nabla \psi^*(\vzh^{(0)}) - \vxh^{(0)}\|\cdot\|\nabla\psi^*(\vz^{(1)}) - \vx^{(1)}\|\notag\\
&- \frac{\sigma}{2A_1}\left(\|\nabla \psi^*(\vzh^{(0)}) - \vxh^{(0)}\|^2 + \|\nabla\psi^*(\vz^{(1)}) - \vx^{(1)}\|^2\right)+\frac{1}{A_1} D_{\psi}(\vx^*, \vxh^{(0)})\notag\\
\leq &\frac{1}{A_1} D_{\psi}(\vx^*, \vxh^{(0)}),\notag
\end{align}
where we have used Proposition \ref{prop:Bregman-divergence-properties}, $\vx^{(1)} = \nabla\psi^*(\vzh^{(0)})$, and $\frac{{a_1}^2}{A_1} = A_1 \leq \frac{\sigma}{L}$.
\end{proof}

\begin{repproposition}{prop:ineq-needed-for-mp-holder}
$$a_{i+1}\innp{\nabla f(\vx^{(i+1)}) - \nabla f(\vxh^{(i)}), \nabla \psi^*(\vz^{(i+1)}) - \nabla \psi^*(\vzh^{(i)})} \leq \sigma^{-1}{a_{i+1}}^2 \|\nabla f(\vx^{(i+1)}) - \nabla f(\vxh^{(i)})\|^2.$$
\end{repproposition}
\begin{proof}
From the first order optimality condition in Fact \ref{fact:danskin}, $\forall \vx, \vy \in X$:
\begin{align}
\innp{\nabla \psi(\nabla \psi^*(\vz^{(i+1)})) - \vz^{(i+1)}, \vx - \nabla \psi^*(\vz^{(i+1)})} &\geq 0, \text{ and} \label{eq:holder-1st-oo-1}\\
\innp{\nabla \psi(\nabla \psi^*(\vzh^{(i)})) - \vzh^{(i)}, \vy - \nabla \psi^*(\vzh^{(i)})} &\geq 0. \label{eq:holder-1st-oo-2}
\end{align}
Letting $\vx = \nabla \psi^*(\vzh^{(i)})$, $\vy = \nabla \psi^*(\vz^{(i+1)})$, and summing (\ref{eq:holder-1st-oo-1}) and (\ref{eq:holder-1st-oo-2}):
\begin{align}
&\innp{\vzh^{(i)} - \vz^{(i+1)},  \nabla \psi^*(\vzh^{(i)}) - \nabla \psi^*(\vz^{(i+1)})}\notag\\
&\hspace{.5in}\geq \innp{\nabla \psi(\nabla \psi^*(\vzh^{(i)})) - \nabla \psi(\nabla \psi^*(\vz^{(i+1)})), \nabla \psi^*(\vzh^{(i)}) - \nabla \psi^*(\vz^{(i+1)})}\notag\\
&\hspace{.5in} \geq \sigma \|\nabla \psi^*(\vzh^{(i)}) - \nabla \psi^*(\vz^{(i+1)})\|^2, \label{eq:holder-psi-sc}
\end{align}
where (\ref{eq:holder-psi-sc}) follows by the $\sigma$-strong convexity of $\psi(\cdot)$. 
Using the Cauchy-Schwartz inequality and dividing both sides by $\|\nabla \psi^*(\vzh^{(i)}) - \nabla \psi^*(\vz^{(i+1)})\|$ gives $\|\vzh^{(i)} - \vz^{(i+1)}\| \geq \sigma\|\nabla \psi^*(\vzh^{(i)}) - \nabla \psi^*(\vz^{(i+1)})\|$. 

Since, by the step definition (\ref{eq:dt-acc-mp-step}), $\vzh^{(i)} - \vz^{(i+1)} = a_{i+1}(\nabla f(\vx^{(i+1)}) - \nabla f(\vxh^{(i)}))$, applying Cauchy-Schwartz Inequality to $a_{i+1}\innp{\nabla f(\vx^{(i+1)}) - \nabla f(\vxh^{(i)}), \nabla \psi^*(\vz^{(i+1)}) - \nabla \psi^*(\vzh^{(i)})}$ completes the proof.
\end{proof}

\end{document}